\documentclass[smallextended,referee,envcountsect,]{svjour3}

\usepackage{bm}
\usepackage{amsfonts,amsmath,amssymb}
\usepackage{enumerate}
\usepackage{enumitem}
\usepackage{esint}
\usepackage{graphicx,psfrag}
\usepackage{hyperref}
\DeclareMathAlphabet{\mathpzc}{OT1}{pzc}{m}{it}

\newcommand{\EO}[1]{{\color{black}#1}}

\usepackage{lineno}
\usepackage{pgf,tikz,pgfplots}
\usepackage{pstricks-add}

\DeclareFontEncoding{FMS}{}{}
\DeclareFontSubstitution{FMS}{futm}{m}{n}
\DeclareFontEncoding{FMX}{}{}
\DeclareFontSubstitution{FMX}{futm}{m}{n}
\DeclareSymbolFont{fouriersymbols}{FMS}{futm}{m}{n}
\DeclareSymbolFont{fourierlargesymbols}{FMX}{futm}{m}{n}
\DeclareMathDelimiter{\VERT}{\mathord}{fouriersymbols}{152}{fourierlargesymbols}{147}
\usetikzlibrary{arrows}

\smartqed
\usepackage{graphicx}
\journalname{JOTA}

\begin{document}

\title{\EO{An Optimal Control Problem for the Navier--Stokes Equations with Point Sources}}

%\subtitle{Using  the  LaTex Template}
%\subtitle{Optimal control for the Navier--Stokes system}

\author{Francisco Fuica, Felipe Lepe, Enrique Ot\'arola, Daniel Quero}

\institute{Francisco Fuica, Enrique Ot\'arola (corresponding author), Daniel Quero  \at
             Universidad T\'ecnica Federico Santa Mar\'ia, Valpara\'iso, Chile\\
             francisco.fuica@sansano.usm.cl, enrique.otarola@usm.cl, daniel.quero@alumnos.usm.cl
           \and
			Felipe Lepe \at
             Universidad del B\'io B\'io, Concepci\'on, Chile\\
              flepe@ubiobio.cl
}

\date{Received: date / Accepted: date}
%The correct dates will be entered by the editor.

\maketitle

\begin{abstract}
We analyze, in two dimensions, an optimal control problem for the Navier--Stokes equations where the control variable corresponds to the amplitude of forces modeled as point sources; control constraints are also considered. This particular setting leads to solutions to the state equation exhibiting reduced regularity properties. We operate under the framework of Muckenhoupt weights, Muckenhoupt-weighted Sobolev spaces, and the corresponding weighted norm inequalities 
and derive the existence of optimal solutions and first and, necessary and sufficient, second order optimality conditions.
\end{abstract}
\keywords{Optimal control problems, Navier--Stokes equations, Dirac measures, Muckenhoupt weights, First and second order optimality conditions.}
\subclass{
35Q30,         % Navier-Stokes equations 
49J20,          % Existence theories for optimal control problems involving partial differential equations
49K20.         % Optimality conditions for problems involving partial differential equations
}

%All acknowledgements should be placed in the back of the paper after Conclusions..

\section{Introduction}\label{sec:intro}
\EO{The purpose of this paper is to study the existence of optimal solutions and first and, necessary and sufficient, second order optimality conditions for an optimal control problem that involves the stationary Navier--Stokes equations.} The control variable corresponds to the amplitude of  forces modeled as point sources supported at some prescribed points of the underlying spatial domain \EO{(Dirac measures)}; control constraints are also considered. The \EO{thus} singular control forcing appears in the right-hand side of the momentum equation. \EO{We notice that,} since Dirac measures are supported 
at points, and points have Lebesgue measure zero, the aforementioned optimization setting can be seen as an instance of \emph{sparse} PDE-constrained optimization \cite{MR2556849,MR2826983,MR3023751,MR2775195}
\EO{and finds} relevance in applications where one can specify the position of actuators at finitely many \EO{prespecified} points. \EO{We mention references \cite{MR2086168} and \cite{MR2525606} for applications within the context of the active control of sound and vibrations, respectively. Regarding analysis, we mention references \cite{MR3225501,MR3800041,MR3878607},} where the corresponding PDE-constrained optimization problem for when the state equation is a Poisson problem is considered. These references also design and analyze some suitable finite element discretizations. Extensions of the theory to the Stokes and semilinear elliptic equations have been recently investigated in \cite{MR4304887} and \cite{Otarola_semilinear_deltas}, respectively.

To the best of our knowledge, the only work available in the literature that considers an optimal control problem for the stationay Navier--Stokes equations with a control that is measure valued is \cite{MR3936891}. Under the assumption that the underlying domain $\Omega \subset \mathbb{R}^{2}$ is of class $C^{2}$, the authors \EO{derive the existence of local solutions for the corresponding optimal control problem and derive necessary and sufficient conditions for local optimality of controls}. In addition, on the basis of a suitable second order condition, the authors prove the stability of optimal states with respect to perturbations of the optimal control problem data.

In \EO{our work} we analyze an optimal control problem for the stationary Navier--Stokes equations with a control variable that corresponds \EO{to the amplitude of forces modeled as point sources. This setting leads to the first difficulty within our analysis: standard energy arguments do not apply to obtain suitable estimates and solutions to the Navier--Stokes equations exhibit reduced regularity properties. In order to deal with such a singular setting, we operate under the framework developed in \cite{MR3998864,MR3906341}, which is based on the theory of Muckenhoupt weights, Muckenhoupt-weighted Sobolev spaces, and weighted norm inequalities. A second difficulty within our analysis is the nonuniqueness of solutions to the Navier--Stokes equations. An assumption guaranteeing local uniqueness of the state equation around optimal controls is thus needed to derive first and second order optimality conditions \cite{MR2338434,MR3936891}. We thus operate under the framework of \emph{regular solutions} (see Definition \ref{def:reg_sol}) \cite{MR2338434,MR3936891,MR4395150}. Note that this framework is satisfied whenever a suitable smallness assumption on controls is fulfilled.} We provide a complete analysis for our optimal control problem that includes existence of optimal solutions (Theorem \ref{thm:existence_optimal_solution}), first order optimality conditions (Theorem \ref{thm:first_opt_cond}), and necessary and sufficient second order optimality conditions (Theorems \ref{thm:second_necess_suff_opt_cond} and \ref{thm:equivalent_opt_cond}). As instrumental results, we analyze \EO{a suitable linearization of the Navier--Stokes equations and the corresponding adjoint state equations in weighted spaces. We also analyze regularity properties for the solution to the adjoint equations. In addition to the difficulties that were previously mentioned, we have} to deal with the fact that solutions to the state and adjoint equations lie in different function spaces. The analysis that we provide thus requires fine properties of Muckenhoupt weights and embeddings between weighted and non-weighted spaces. This subtle intertwining of ideas is one of the highlights of our contribution.

\EO{The contents of our manuscript are organized as follows. In Sect. \ref{sec:model} we introduce the PDE-constrained optimization problem that is under consideration. We collect background information and the main assumptions under which we shall operate in Sect. \ref{sec:notation_and_prel}. Here, we also introduce the concept of \emph{regular solution} for the Navier--Stokes equations and prove that an operator associated to the linearization of such a system is an isomorphism on suitable weighted spaces. In Sect. \ref{sec:ocp} we introduce a weak formulation for our optimal control problem and prove the existence of solutions. Sect. \ref{sec:op_conditions} is dedicated to the analysis of optimality conditions:  we derive first and, necessary and sufficient, second order optimality conditions. We conclude our work with Sect. \ref{sec:conclusions}, where we provide a brief summary of the obtained results and comment on possible extensions.}

\vspace{-0.3cm}
%%%%%%%%%%%%%%%%%%%%%%%%%%%%%%%%%%%%%%%%%%%%%%%%%%%%%%%%%%%%
%%%%%%%%%%%%%%%%%%%%%%%%%%%%%%%%%%%%%%%%%%%%%%%%%%%%%%%%%%%%
%%%%%%%%%%%%%%%%%%%%%%%%%%%%%%%%%%%%%%%%%%%%%%%%%%%%%%%%%%%%
%%%%%%%%%%%%%%%%%%%%%%%%%%%%%%%%%%%%%%%%%%%%%%%%%%%%%%%%%%%%
%%%%%%%%%%%%%%%%%%%%%%%%%%%%%%%%%%%%%%%%%%%%%%%%%%%%%%%%%%%%
%%%%%%%%%%%%%%%%%%%%%%%%%%%%%%%%%%%%%%%%%%%%%%%%%%%%%%%%%%%%
%%%%%%%%%%%%%%%%%%%%%%%%%%%%%%%%%%%%%%%%%%%%%%%%%%%%%%%%%%%%
%%%%%%%%%%%%%%%%%%%%%%%%%%%%%%%%%%%%%%%%%%%%%%%%%%%%%%%%%%%%
\section{Statement of the Problem}
\label{sec:model}
To describe our problem, we let $\Omega\subset\mathbb{R}^2$ be an open and bounded \EO{domain} with Lipschitz boundary $\partial\Omega$ and let $\emptyset \neq \mathcal{D} \subset \Omega$ be a finite ordered set with cardinality $\#\mathcal{D} =: \ell$. Given a desired velocity \EO{field} $\mathbf{y}_\Omega\in \mathbf{L}^2(\Omega)$ and a regularization parameter $\eta>0$, we introduce the cost functional 
\begin{equation}\label{def:cost_functional}
J(\mathbf{y},\mathcal{U}):=\frac{1}{2}\|\mathbf{y}-\mathbf{y}_\Omega\|_{\mathbf{L}^2(\Omega)}^{2}+\frac{\eta}{2}\sum_{t\in\mathcal{D}}|\mathbf{u}_t|^2, 
\quad \mathcal{U}= (\mathbf{u}_{1},\ldots,\mathbf{u}_{\ell}),
\quad \mathbf{u}_{t} \in \mathbb{R}^2.
\end{equation} 

\EO{The PDE-constrained optimization problem under consideration reads as follows: Find $\min J(\mathbf{y},\mathcal{U})$ subject to the stationary Navier--Stokes equations}
\begin{equation}\label{def:state_eq}
-\nu\Delta \mathbf{y}+(\mathbf{y}\cdot\nabla)\mathbf{y}+\nabla p =   \sum_{t\in\mathcal{D}}\mathbf{u}_t \delta_t \text{ in } \Omega, 
\,\,\,\,\,
\text{div }\mathbf{y}=0 \text{ in } \Omega,
\,\,\,\,\,
\mathbf{y}=\mathbf{0}  \text{ on } \partial\Omega,
\end{equation}
and the control constraints 
\begin{equation}\label{def:box_constraints}
\mathcal{U}\in \mathbb{U}_{ad},
\
\mathbb{U}_{ad}:=\{\mathcal{V}=(\mathbf{v}_1,\ldots,\mathbf{v}_{\ell}) \in [\mathbb{R}^2]^{\ell}:  \mathbf{a}_t \leq \mathbf{v}_t \leq \mathbf{b}_t \text{ for all } t\in \mathcal{D} \},
\end{equation}
with $\mathbf{a}_t,\mathbf{b}_t \in \mathbb{R}^2$ satisfying $\mathbf{a}_t < \mathbf{b}_t$ for every $t\in\mathcal{D}$. We immediately comment that, 
throughout this work, vector inequalities must be understood componentwise and that $|\cdot|$ denotes the \EO{euclidean} norm in $\mathbb{R}^2$. In \eqref{def:state_eq}, $\mathbf{y}$ represents the velocity of the fluid, $p$ represents the pressure, $\nu>0$ denotes the kinematic viscosity, and $\delta_t$ corresponds to the Dirac delta supported at the interior point $t \in \mathcal{D}$.
\section{Notation and Preliminaries}\label{sec:notation_and_prel}

The main purpose of this section is to \EO{introduce} the main notation and \EO{recall} basic results which we shall use later on. 
%%%%%%%%%%%%%%%%%%%%%%%%%%%%%%%%%%%%%%%%%%%%%%%%%%%%%%%%%%%%
%%%%%%%%%%%%%%%%%%%%%%%%%%%%%%%%%%%%%%%%%%%%%%%%%%%%%%%%%%%%
%%%%%%%%%%%%%%%%%%%%%%%%%%%%%%%%%%%%%%%%%%%%%%%%%%%%%%%%%%%%
%%%%%%%%%%%%%%%%%%%%%%%%%%%%%%%%%%%%%%%%%%%%%%%%%%%%%%%%%%%%
%%
\subsection{Notation}\label{sec:notation}
\EO{Let $\mathfrak{X}$ be a Banach function space. We denote by $\mathfrak{X}'$, $\mathfrak{X}''$, and $\|\cdot\|_{\mathfrak{X}}$ the dual, the bidual, and the norm of $\mathfrak{X}$, respectively. Let $\{ x_n \}_{n\in \mathbb{N}}$ be a sequence in $\mathfrak{X}$.
We denote by $x_n \rightarrow x$ and $x_n \rightharpoonup x$ the strong and weak convergence, respectively, of $\{ x_n \}_{n\in \mathbb{N}}$ to $x$ in $\mathfrak{X}$.
We denote by $\langle \cdot, \cdot\rangle_{\mathfrak{X}',\mathfrak{X}}$ the duality pairing between $\mathfrak{X}'$ and $\mathfrak{X}$ and simply write $\langle \cdot, \cdot\rangle$ when $\mathfrak{X}'$ and $\mathfrak{X}$ are clear from the context. We write $\mathfrak{X}\hookrightarrow\mathfrak{Y}$ to denote that $\mathfrak{X}$ is continuously embedded in the Banach function space $\mathfrak{Y}$.}

\EO{Let $E \subset \mathbb{R}^2$ be a Lebesgue measurable set. We denote the Lebesgue measure of such a set by $|E|$ . For $f: E \rightarrow \Omega$, we set} 
\[
\fint_{E} f = \frac{1}{|E|} \int_E f.
\]
%We shall use standard notation for weighted Lebesgue and Sobolev spaces. 
%To denote vector-valued functions we shall use lowercase bold letters, whereas to denote function spaces we shall use uppercase bold letters. 
%For instance, for a bounded domain $G \subset \mathbb{R}^{d}$, with $d\in\{1,2\}$, we denote $\mathbf{L}^2(G)=[L^2(G)]^2$ and equip $\mathbf{L}^2(G)$ with the following inner product and norm, respectively,
%\begin{equation*}
%(\mathbf{w},\mathbf{v})_{\mathbf{L}^2(G)}=\int_G \mathbf{w}\cdot \mathbf{v}, 
%\qquad
%\|\mathbf{v}\|_{\mathbf{L}^2(G)}=(\mathbf{v},\mathbf{v})_{\mathbf{L}^2(G)}^{\frac{1}{2}}\qquad \forall\: \mathbf{w},\mathbf{v}\in \mathbf{L}^2(G).
%\end{equation*}
%
%Given $q \in (1,\infty)$, we denote by $q'$ its H\"older conjugate, i.e., the real number such that $1/q+ 1/q'= 1$. 

By $a \lesssim b$ we mean $a \leq C b$, with a positive constant $C$ \EO{that does not depend on either $a$ or $b$. The value of $C$ might change at each occurrence}. If the particular value of the constant $C$ is of relevance for our analysis, we will thus assign it a name.

%%%%%%%%%%%%%%%%%%%%%%%%%%%%%%%%%%%%%%%%%%%%%%%%%%%%%%%%%%%%
%%%%%%%%%%%%%%%%%%%%%%%%%%%%%%%%%%%%%%%%%%%%%%%%%%%%%%%%%%%%
%%%%%%%%%%%%%%%%%%%%%%%%%%%%%%%%%%%%%%%%%%%%%%%%%%%%%%%%%%%%
%%%%%%%%%%%%%%%%%%%%%%%%%%%%%%%%%%%%%%%%%%%%%%%%%%%%%%%%%%%%

\subsection{Muckenhoupt Weights}\label{sec:wse}

By a weight, we shall mean a locally integrable function $\omega$ on $\mathbb{R}^2$ such that $\omega(x)>0$ for a.e.~$x\in \mathbb{R}^2$. A special class of weights that will be of importance for our analysis is the so-called Muckenhoupt class $A_{2}$ \cite{Javier2001,Fabes_et_al1982,MR0293384,Turesson2000} .
\begin{definition}[Muckenhoupt Class $A_2$]\label{def:mucken_weight}
A weight $\omega$ belongs to the Muckenhoupt class $A_{2}$ if
\begin{equation*}
[\omega]_{A_2} :=
\sup_{B}\left( \fint_B\omega \right)\left(\fint_B\omega^{-1} \right) < \infty,
\end{equation*}
where the supremum is taken over all balls $B$ in $\mathbb{R}^{2}$. We call $[\omega]_{A_2}$ the Muckenhoupt characteristic of $\omega$.
\end{definition}

\EO{We refer the interested reader to \cite{Javier2001,Fabes_et_al1982,MR0293384,Turesson2000} for basic facts about the Muckenhoupt class $A_{2}$}. To present prototypical examples of Muckenhoupt weights, we let $\mathcal{K}$ be a smooth compact submanifold of dimension $k \in \{ 0,1\}$ and define $\mathrm{d}_{\mathcal{K}}^{\alpha}(x):= \mathrm{dist}(x,\mathcal{K})^{\alpha}$. The weight $\mathrm{d}_{\mathcal{K}}^{\alpha}$ belongs to the Muckenhoupt class $A_2$ provided $\alpha \in (-(2-k),(2-k))$; see \cite{ACDT2014} and \cite[Lemma 2.3(vi)]{MR1601373}. \EO{We thus identify the following two particular cases:}
%The aforementioned prototypical examples of Muckenhoupt weights are thus in order:
\begin{enumerate}
\item 	Let $z \in \Omega$. Then, the weight $\mathrm{d}_{z}^{\alpha} \in A_2$ if $\alpha \in (-2,2)$.
\item Let $\gamma \subset \Omega$ be a smooth closed curve without self-intersections. Then, the weight $\mathrm{d}_{\gamma}^{\alpha} \in A_2$ if $\alpha \in (-1,1)$.
\end{enumerate}

As a consequence of the fact that the lower dimensional objects $z$ and $\gamma$ are strictly contained in $\Omega$, there \EO{are} neighborhoods of $\partial \Omega$ where \EO{the} weights $\mathrm{d}_{z}^{\alpha}$ and $\mathrm{d}_{\gamma}^{\alpha}$ have no degeneracies or singularities. This simple observation motivates the following restricted class of Muckenhoupt weights \cite[Definition 2.5]{MR1601373}.
\begin{definition}[Class $A_2(G)$]
Let $G\subset\mathbb{R}^2$ be a Lipschitz domain. We say that $\omega\in A_2$ belongs to $A_2(G)$ if there is an open set $\mathcal{G} \subset G$ and $\varepsilon, \omega_l >0$ such that
$
\{x\in G: \textnormal{dist}(x,\partial G)<\varepsilon\}\subset \mathcal{G},
$
$\omega\in C(\bar{\mathcal{G}})$, and $\omega(x) \geq \omega_l$ for all $x\in\bar{\mathcal{G}}$.
\end{definition}

\subsection{Muckenhoupt-weighted Sobolev Spaces}\label{sec:wss}

Let $\omega \in A_2$ and $G\subset \mathbb{R}^{2}$ be an open set. We define 
%\EO{the following weighted Lebesgue and Sobolev spaces:}
\begin{itemize}
\item[$\bullet$] $L^{2}(\omega,G)$ as the space of measurable functions $f$ on $G$ such that 
$$\| f \|_{L^2(\omega,G)}:= \left( \int_G |f|^2 \omega \right)^{\frac{1}{2}} < \infty,$$
\item[$\bullet$] $H^1(\omega,G) := \{f \in L^{2}(\omega,G) : D^{\alpha} f \in L^2(\omega,G) \text{ for } |\alpha| \leq 1 \}$, \EO{endowed with the norm $\| f \|_{H^1(G)} := (  \| f \|_{L^2(G)}^2 + \| \nabla f \|^2_{L^2(G)} )^{\frac{1}{2}}$}, and
\item[$\bullet$] $H^{1}_{0}(\omega,G)$ as the closure of $C_{0}^{\infty}(G)$ in $H^1(\omega,G)$.
%, and
%\item[$\bullet$] \DQ{$H^{1}_{0}(\omega,G)'$ as the dual space of $H^{1}_{0}(\omega,G)$.}
\end{itemize}
For basic properties of these spaces, such as approximation by smooth functions, extensions theorems, \EO{and interpolation inequalities,} we refer the interested reader to \cite[Chapter 2]{Turesson2000}.

\EO{Spaces of vector valued functions will be denoted by boldface}
%To denote the vector-valued counterpart of these spaces we shall use uppercase bold 
uppercase letters whereas lowercase bold letters will be used to denote vector valued functions. \EO{In particular, we introduce $\mathbf{H}^{1}_{0}(\omega,G)$ and $| \cdot |_{\mathbf{H}^{1}(\omega,G)}$ as follows:}
\begin{equation*}
\EO{\mathbf{H}^{1}_{0}(\omega,G):=[H^{1}_{0}(\omega,G)]^2,
\quad
| \mathbf{v}|_{\mathbf{H}^{1}(\omega,G)}^{2}:=
\|\nabla \mathbf{v} \|_{\mathbf{L}^{2}(\omega,G)}^{2} = \sum_{i=1}^{2}\|\nabla v_{i}\|_{L^{2}(\omega,G)}^2,}
\end{equation*}
\EO{for every $\mathbf{v} \in \mathbf{H}^{1}_{0}(\omega,G)$.}

%Due to the fact that the weight $\omega$ belongs to $A_2$, most of the properties of classical Sobolev spaces have a weighted counterpart. In particular, $L^{2}(\omega,G)$ and $H^1(\omega,G)$ are Hilbert spaces \cite[Proposition 2.1.2]{Turesson2000} and smooth functions are dense \cite[Corollary 2.1.6]{Turesson2000}. In view of a weighted Poincar\'e inequality, that follows from \cite[Theorem 1.3]{Fabes_et_al1982}, we have that in $H^{1}_{0}(\omega,G)$ the seminorm $\|\nabla v\|_{L^2(\omega,G)}$ is equivalent to $\|v\|_{H^1(\omega,G)}$. 
%
%%We introduce the space $\mathbf{H}^{1}_{0}(\omega,G)$ and the semi-norm $| \cdot |_{\mathbf{H}^{1}(\omega,G)}$ as follows:
%%\begin{equation*}
%%\mathbf{H}^{1}_{0}(\omega,G):=[H^{1}_{0}(\omega,G)]^2,
%%\quad
%%| \mathbf{v}|_{\mathbf{H}^{1}(\omega,G)}^{2}:=
%%\|\nabla \mathbf{v} \|_{\mathbf{L}^{2}(\omega,G)}^{2} = \sum_{i=1}^{2}\|\nabla v_{i}\|_{L^{2}(\omega,G)}^2.
%%\end{equation*}
%

\subsection{Weighted Inequalities and Embeddings}

The following fundamental result, which is known as \emph{reverse H\"older inequality}, 
will be \EO{essential} for our analysis; see \cite[Theorem 7.4]{Javier2001}.

\begin{proposition}[Reverse H\"older Inequality]
If $\omega \in A_2$, then there exists \EO{a positive constant} $\epsilon$ such that, for every ball $B \subset \mathbb{R}^2$, we have
\[
%\frac{1}{|B|} \int_{B} \omega^{1+\epsilon} \lesssim \left( \frac{1}{|B|} \int_B \omega \right)^{1+\epsilon}.
\fint_{B} \omega^{1+\epsilon} \lesssim \left(  \fint_B \omega \right)^{1+\epsilon}.
\] 
The hidden constant only depends on the Muckenhoupt characteristic $[\omega]_{A_2}$.
\label{pro:reverse}
\end{proposition}

We now present the following embedding results.
\begin{theorem}[Continuous Embeddings]\label{thm:embedding_result}
If $\omega \in A_{2}$, then there exists $\varepsilon > 0$ such that $ \mathbf{H}_0^1(\omega,\Omega) \hookrightarrow \mathbf{L}^{2 + \varepsilon}(\Omega)$ and there exists $\kappa > 1$ such that $ \mathbf{H}_0^1(\omega,\Omega) \hookrightarrow \mathbf{W}_{0}^{1,\kappa}(\Omega)$.
\end{theorem}

\begin{proof}
We prove the first embedding result; the second one follows from similar considerations. Let $\omega \in A_{2}$ and $\boldsymbol\Phi \in \mathbf{H}_0^1(\omega,\Omega)$. An application of \cite[Theorem 1.3]{Fabes_et_al1982} implies that $\boldsymbol \Phi \in \mathbf{L}^{4}(\omega,\Omega)$. We thus invoke H\"older's inequality to obtain
\begin{equation}
\int_{\Omega}|\boldsymbol\Phi|^{2 + \varepsilon} 
= 
\int_{\Omega}|\boldsymbol\Phi|^{2 + \varepsilon}\omega^{\frac{2+\varepsilon}{4}}\omega^{-\frac{2+\varepsilon}{4}} 
\leq \left(\int_{\Omega}|\boldsymbol\Phi|^{4}\omega\right)^{\frac{2+\varepsilon}{4}}\left(\int_{\Omega}\omega^{-\frac{2+\varepsilon}{2-\varepsilon}}\right)^{\frac{2-\varepsilon}{4}}
\label{eq:vare_L2_H1omega}
\end{equation}
for some $\varepsilon>0$. We observe that, since $\omega \in A_{2}$ and $\frac{2+\varepsilon}{2-\varepsilon} = 1 + \delta$, with $\delta = \frac{2\varepsilon}{2 - \varepsilon}$, the reverse H\"older inequality of Proposition \ref{pro:reverse} allows us to obtain
\[ 
\int_{\Omega}\omega^{-\frac{2+\varepsilon}{2-\varepsilon}}
=
\int_{\Omega}\omega^{-(1+\delta)}
 \lesssim |\Omega|^{-\delta} \left( \int_{\Omega}\omega^{-1} \right)^{1+\delta}
 =
\EO{|\Omega| \left( \fint_{\Omega}\omega^{-1} \right)^{\frac{2+\varepsilon}{2-\varepsilon}}.}
\]
Here, $\varepsilon>0$ is sufficiently small such that \EO{the} previously defined parameter $\delta$ is less or equal that the one dictated by the reverse H\"older inequality. Since $\int_{\Omega}\omega^{-1}$ is uniformly bounded, the previous bound combined with estimate \eqref{eq:vare_L2_H1omega} allow us to conclude.
\qed
\end{proof}
%%%%%%%%%%%%%%%%%%%%%%%%%%%%%%%%%%%%%%%%%%%%%%%%%%%%%%%%%%%%
%%%%%%%%%%%%%%%%%%%%%%%%%%%%%%%%%%%%%%%%%%%%%%%%%%%%%%%%%%%%
%%%%%%%%%%%%%%%%%%%%%%%%%%%%%%%%%%%%%%%%%%%%%%%%%%%%%%%%%%%%
%%%%%%%%%%%%%%%%%%%%%%%%%%%%%%%%%%%%%%%%%%%%%%%%%%%%%%%%%%%%
\subsubsection{A Particular Weight}
\label{sec:particular_weight}

In this section, we introduce a particular weight in the class $A_2$ that will be of fundamental importance. %to analyze our optimal control problem. 
With the finite set $\mathcal{D} \subset \Omega$ at hand, we define
\begin{equation}
\label{eq:d_D}
d_{\mathcal{D}} := \left\{\begin{array}{ll}
\mathrm{dist}(\mathcal{D},\partial \Omega), & \mbox{if } \ell=1,
\\
\min \left \{ \mathrm{dist}(\mathcal{D},\partial \Omega), \min\{|t-t'|: t,t' \in \mathcal{D}, \ t\neq t' \} \right \}, & \mbox{otherwise}.
\end{array}\right.
\end{equation}
We recall that $\ell = \# \mathcal{D}$. Since $\mathcal{D} \subset \Omega$ is finite, we immediately conclude that $d_{\mathcal{D}}>0$. With this notation at hand, we define the weight $\rho$ as follows: 
\begin{equation}\label{def:weight_rho}
\text{If }\ell=1,~ \rho(x) = \mathsf{d}_t^\alpha(x),
\text{ otherwise, }
\rho(x)=
\left\{\begin{array}{lr}
\mathsf{d}_t^\alpha(x), & \exists t\in \mathcal{D}:\mathsf{d}_t(x) < \frac{d_\mathcal{D}}{2},\\
1, & \mathsf{d}_t(x)\geq \frac{d_\mathcal{D}}{2} \: \forall t\in \mathcal{D},
\end{array}
\right.
\end{equation}
where $\mathsf{d}_t(x) := |x - t|$ and $\alpha \in (0,2)$. Since $(0,2) \subset (-2,2)$, owing to  \cite[Theorem 6]{ACDT2014} and \cite[Lemma 2.3 (vi)]{MR1601373}, $\rho \in A_{2}$. The extra restriction on $\alpha$, namely, $\alpha>0$, is needed in order to guarantee that for $t\in\mathcal{D}$ and $\mathbf{v}_t \in \mathbb{R}^2$, $\mathbf{v}_t\delta_{t} \in \mathbf{H}_0^1(\rho^{-1},\Omega)'$; see \cite[Remark 21.19]{MR2305115} and \cite[Proposition 5.2]{MR4081912} for details.

\EO{The following lemma provides instrumental embedding and density results.}

\begin{lemma}[Embedding and Density Results]\label{lemma:density_H-1_dualH01_omega-1} 
\EO{Let $\rho$ be the weight defined in \eqref{def:weight_rho}. If $\alpha \in (0,2)$, then}
\begin{itemize}
\EO{\item[\textnormal{(i)}] $\mathbf{H}^{1}_{0}(\rho^{-1},\Omega) \hookrightarrow \mathbf{H}_{0}^{1}(\Omega) \hookrightarrow \mathbf{H}^{1}_{0}(\rho,\Omega)$,}
\EO{\item[\textnormal{(ii)}] $\mathbf{H}^{1}_{0}(\rho,\Omega)' \hookrightarrow \mathbf{H}^{-1}(\Omega) \hookrightarrow \mathbf{H}^{1}_{0}(\rho^{-1},\Omega)'$, and}
\EO{\item[\textnormal{(iii)}] $\mathbf{H}^{-1}(\Omega)$ is dense in $\mathbf{H}^{1}_{0}(\rho^{-1},\Omega)'$.}
\end{itemize}
\end{lemma}
\begin{proof}
\begin{itemize}
\item[(i)] \EO{We prove that $\mathbf{H}^{1}_{0}(\rho^{-1},\Omega) \hookrightarrow \mathbf{H}_{0}^{1}(\Omega)$; the other embedding follows from similar considerations. If $\mathbf{v} \in \mathbf{H}_{0}^{1}(\rho^{-1},\Omega)$, then}
\begin{equation*}
\EO{
\|\nabla \mathbf{v}\|_{\mathbf{L}^{2}(\Omega)} 
= 
\|\rho^{\frac{1}{2}}\rho^{-\frac{1}{2}}\nabla \mathbf{v}\|_{\mathbf{L}^{2}(\Omega)} 
\leq 
\|\rho^{\frac{1}{2}}\|_{L^{\infty}(\Omega)}\|\nabla \mathbf{v}\|_{\mathbf{L}^{2}(\rho^{-1},\Omega)}.
}
\end{equation*}
\EO{Notice that, since $\alpha >0$, the weight $\rho$ is uniformly bounded in $\Omega$.}
\item[(ii)] \EO{We prove that $\mathbf{H}^{-1}(\Omega) \hookrightarrow \mathbf{H}^{1}_{0}(\rho^{-1},\Omega)'$; the other embedding follows from similar considerations. Let $\mathbf{L}$ be an arbitrary element in $\mathbf{H}^{-1}(\Omega)$. In view of the embedding $\mathbf{H}^{1}_{0}(\rho^{-1},\Omega) \hookrightarrow \mathbf{H}^{1}_{0}(\Omega)$, we immediately deduce that
\begin{align*}
\|\mathbf{L}\|_{\mathbf{H}^{1}_{0}(\rho^{-1},\Omega)'}  
&
:= \sup_{ \mathbf{v} \in \mathbf{H}^{1}_{0}(\rho^{-1},\Omega)} \frac{ \langle \mathbf{L} , \mathbf{v} \rangle  }{ \|\nabla \mathbf{v}\|_{\mathbf{L}^{2}(\rho^{-1},\Omega)}} 
\\
& \leq \sup_{ \mathbf{v} \in \mathbf{H}^{1}_{0}(\Omega)} \frac{ \langle \mathbf{L} , \mathbf{v} \rangle  }{ \|\nabla \mathbf{v}\|_{\mathbf{L}^{2}(\rho^{-1},\Omega)}}
\leq \|\rho^{\frac{1}{2}}\|_{L^{\infty}(\Omega)} \|\mathbf{L}\|_{\mathbf{H}^{-1}(\Omega)}.
\end{align*}
This implies that $\mathbf{L} \in \mathbf{H}^{1}_{0}(\rho^{-1},\Omega)'$, as we intended to show.}
\item[(iii)] \EO{For completeness, we provide a proof based on \cite[Corollary 1.8 and Remark 5]{Brezis}: Let $\mathbf{F} \in \mathbf{H}_{0}^{1}(\rho^{-1},\Omega)''$ be such that
\begin{equation}\label{eq:density_prove_1}
\langle \mathbf{F}, \mathbf{L} \rangle_{\mathbf{H}_{0}^{1}(\rho^{-1},\Omega)'',\mathbf{H}_{0}^{1}(\rho^{-1},\Omega)'} = 0 \quad \forall \mathbf{L} \in \mathbf{H}^{-1}(\Omega).
\end{equation}
We have to prove that $\mathbf{F} = \mathbf{0}$ in $\mathbf{H}_{0}^{1}(\rho^{-1},\Omega)''$. Since $\mathbf{H}_{0}^{1}(\rho^{-1},\Omega)$ is a reflexive space, there exists $\mathbf{f} \in \mathbf{H}_{0}^{1}(\rho^{-1},\Omega)$ such that 
\begin{equation}\label{eq:density_prove_2}
\langle \mathbf{F}, \mathbf{Q} \rangle_{\mathbf{H}_{0}^{1}(\rho^{-1},\Omega)'',\mathbf{H}_{0}^{1}(\rho^{-1},\Omega)'} = \langle \mathbf{Q}, \mathbf{f} \rangle_{\mathbf{H}_{0}^{1}(\rho^{-1},\Omega)',\mathbf{H}_{0}^{1}(\rho^{-1},\Omega)}
\end{equation}
for all $\mathbf{Q} \in \mathbf{H}_{0}^{1}(\rho^{-1},\Omega)'$. In view of $\mathbf{H}^{-1}(\Omega) \hookrightarrow \mathbf{H}^{1}_{0}(\rho^{-1},\Omega)'$, relation \eqref{eq:density_prove_2} is also valid for all $\mathbf{Q} \in \mathbf{H}^{-1}(\Omega)$. On the other hand, the Riesz representation theorem immediately yields that for every $\mathbf{u} \in \mathbf{H}_{0}^{1}(\Omega)$ there exists $\mathbf{L} \in \mathbf{H}^{-1}(\Omega)$ satisfying
\begin{equation}\label{eq:Riesz_representative}
\langle \mathbf{L},\mathbf{w} \rangle_{\mathbf{H}^{-1}(\Omega),\mathbf{H}_{0}^{1}(\Omega)} = \int_{\Omega}\nabla \mathbf{u}: \nabla \mathbf{w} \quad \forall \mathbf{w} \in \mathbf{H}_{0}^{1}(\Omega).
\end{equation}
Since $\mathbf{H}_{0}^{1}(\rho^{-1},\Omega)\hookrightarrow\mathbf{H}_{0}^{1}(\Omega)$, identity \eqref{eq:Riesz_representative} holds for $\mathbf{u} \in \mathbf{H}_{0}^{1}(\rho^{-1},\Omega)$. In particular, for $\mathbf{f} \in \mathbf{H}_{0}^{1}(\rho^{-1},\Omega)$, there exists $\mathbf{L}_{\mathbf{f}} \in \mathbf{H}^{-1}(\Omega)$ such that
$$  \langle \mathbf{L}_{\mathbf{f}},\mathbf{w} \rangle_{\mathbf{H}^{-1}(\Omega),\mathbf{H}_{0}^{1}(\Omega)} = \int_{\Omega}\nabla \mathbf{f}: \nabla \mathbf{w} \quad \forall \mathbf{w} \in \mathbf{H}_{0}^{1}(\Omega). $$
Set $\mathbf{w} = \mathbf{f}$ into the previous relation and invoke \eqref{eq:density_prove_1} and \eqref{eq:density_prove_2} to conclude that $\|\nabla \mathbf{f}\|_{\mathbf{L}^{2}(\Omega)} = 0$. This implies that $\mathbf{f} = \mathbf{0}$ a.e. in $\Omega$. Consequently, $\mathbf{f} = \mathbf{0}$ in $\mathbf{H}_{0}^{1}(\rho^{-1},\Omega)$ and hence $\mathbf{F} = \mathbf{0}$ in $\mathbf{H}_{0}^{1}(\rho^{-1},\Omega)''$.
\qed}
\end{itemize} 

\end{proof}

%%%%%%%%%%%%%%%%%%%%%%%%%%%%%%%%%%%%%%%%%%%%%%%%%%%%%%%%%%%%
%%%%%%%%%%%%%%%%%%%%%%%%%%%%%%%%%%%%%%%%%%%%%%%%%%%%%%%%%%%%
%%%%%%%%%%%%%%%%%%%%%%%%%%%%%%%%%%%%%%%%%%%%%%%%%%%%%%%%%%%%
%%%%%%%%%%%%%%%%%%%%%%%%%%%%%%%%%%%%%%%%%%%%%%%%%%%%%%%%%%%%

\section{The Navier--Stokes Equations Under Singular Forcing}
\label{sec:NS_singular_rhs}

\EO{In this section, we follow the weighted approach developed in \cite{MR3998864} and review existence results for a suitable variational formulation of the stationary Navier--Stokes equations under singular forcing. By singular, we mean that the forcing term of the momentum equation is allowed to belong to the space $\mathbf{H}_0^1(\omega^{-1},\Omega)'$ with $\omega \in A_2$. To be precise, given $\mathbf{f} \in \mathbf{H}_0^1(\omega^{-1},\Omega)'$, we consider the following weak problem: Find $(\boldsymbol\Phi,\zeta)\in \mathbf{H}_0^1(\omega,\Omega)\times L^2(\omega,\Omega)/\mathbb{R}$ such that 
\begin{equation}
\label{eq:weak_NS_delta}
\begin{aligned}
\int_{\Omega} \left(\nu \nabla \boldsymbol\Phi : \nabla \mathbf{v} - \boldsymbol\Phi\otimes\boldsymbol\Phi: \nabla \mathbf{v} - \zeta\textrm{div }\mathbf{v}\right) 
&
= \langle \mathbf{f},\mathbf{v}\rangle_{\mathbf{H}_0^1(\omega^{-1},\Omega)',\mathbf{H}_0^1(\omega^{-1},\Omega)}, \\
 \int_{\Omega}  q \text{div }\boldsymbol\Phi
 &
 = 0,
\end{aligned}
\end{equation}
for all $(\mathbf{v},q)\in\mathbf{H}_0^1(\omega^{-1},\Omega)\times L^2(\omega^{-1},\Omega)/\mathbb{R}$. Here, $\nu>0$ and $\omega \in A_{2}$}.
%, $\omega$ denotes a weight in $A_2$, and $\langle \cdot, \cdot\rangle$ denotes the duality pairing between $\mathbf{H}_0^1(\omega^{-1},\Omega)'$ and $\mathbf{H}_0^1(\omega^{-1},\Omega)$.} 
%The spaces $\mathcal{X}$ and $\mathcal{Y}$ are defined as follows:
%\begin{equation}\label{eq:spaces_X_Y}
%\mathcal{X}:=\mathbf{H}_0^1(\omega,\Omega)\times L^2(\omega,\Omega)/\mathbb{R}, \qquad \mathcal{Y}:=\mathbf{H}_0^1(\omega^{-1},\Omega)\times L^2(\omega^{-1},\Omega)/\mathbb{R}.
%\end{equation}

\EO{Existence of solutions without smallness conditions is as follows \cite[Theorem 1]{MR3998864}: Let $\Omega$ be Lipschitz, $\omega \in A_2(\Omega)$, $\nu >0$, and $\mathbf{f} \in \mathbf{H}_0^1(\omega^{-1},\Omega)'$. Thus, \eqref{eq:weak_NS_delta} has at least one solution $(\boldsymbol\Phi, \zeta) \in \mathbf{H}_0^1(\omega,\Omega)\times L^2(\omega,\Omega)/\mathbb{R}$, which satisfies}
\begin{equation}\label{eq:NS_weighted_stab}
\|\nabla \boldsymbol\Phi\|_{\mathbf{L}^{2}(\omega,\Omega)} + \|\zeta\|_{L^{2}(\omega,\Omega)}
\lesssim
\|\mathbf{f}\|_{\mathbf{H}_0^1(\omega^{-1},\Omega)'}.
\end{equation}

\EO{
A similar result can be obtained on $L^{p}$-based spaces. In an abuse of notation, we denote by $(\boldsymbol\Phi, \zeta) \in \mathbf{W}_0^{1,p}(\Omega)\times L^{p}(\Omega)/\mathbb{R}$ the solution to
\begin{equation}
\label{eq:weak_NS_delta_p}
\begin{aligned}
\int_{\Omega} \left(\nu \nabla \boldsymbol\Phi : \nabla \mathbf{v} - \boldsymbol\Phi\otimes\boldsymbol\Phi: \nabla \mathbf{v} - \zeta\textrm{div }\mathbf{v}\right) 
&
= \langle \mathbf{f},\mathbf{v}\rangle_{\mathbf{W}^{-1,p}(\Omega),\mathbf{W}^{1,p'}(\Omega)}, \\
\int_{\Omega}  q \text{div }\boldsymbol\Phi
&
= 0,
\end{aligned}
\end{equation}
for all $(\mathbf{v},q) \in \mathbf{W}_0^{1,p'}(\Omega) \times L^{p'}(\Omega) \setminus \mathbb{R}$. Here, $\mathbf{f} \in \mathbf{W}^{-1,p}(\Omega)$ and $p'$ is such that $1/p + 1/p' = 1$. Let us assume that $\Omega$ is Lipschitz and that $\nu>0$. Within this setting at hand, we have the following existence result: If $p \in (4/3-\epsilon, 2)$, where $\epsilon = \epsilon (\Omega)>0$ denotes a constant that depends on $\Omega$, then problem \eqref{eq:weak_NS_delta_p} has at least one solution $(\boldsymbol\Phi, \zeta) \in \mathbf{W}_0^{1,p}(\Omega)\times L^{p}(\Omega)/\mathbb{R}$. In addition, we have the following stability bound:
\begin{equation}\label{eq:NS_p_stab}
\EO{\|\nabla \boldsymbol\Phi\|_{\mathbf{L}^{p}(\Omega)} + \|\zeta\|_{L^{p}(\Omega)}
\lesssim
\|\mathbf{f}\|_{\mathbf{W}^{-1,p}(\Omega)}.}
\end{equation}
The proof of such an existence result and the stability bound \eqref{eq:NS_p_stab} follows from the arguments elaborated in \cite[Section 3]{MR4257071}.}

%\begin{remark}[ $\mathbf{W}_{0}^{1,p}$--regularity of $\boldsymbol{\Phi}$]\label{rem:reg_state_vel}
%Let $\mathcal{V} = (\mathbf{v}_{1},\ldots,\mathbf{v}_{\ell}) \in [\mathbb{R}^{2}]^{\ell}$ and $\alpha \in (0,2)$. Since $\delta_{t} \in \mathbf{H}_0^1(\rho^{-1},\Omega)' \cap \mathbf{W}^{-1,p}(\Omega)$, for all $t \in \mathcal{D}$ and for all $p < 2$, it follows that, when the forcing term $\mathbf{f} = \sum_{t \in \mathcal{D}}\mathbf{v}_{t}\delta_{t}$, the velocity field $\boldsymbol\Phi$ that solves \eqref{eq:weak_NS_delta} belongs to $\mathbf{H}_{0}^{1}(\rho,\Omega) \cap \mathbf{W}_{0}^{1,p}(\Omega)$, for all $p \in (4/3 - \lambda, 2)$, where $\lambda = \lambda(\Omega) > 0$. The latter follows by adapting the existence results of \cite[Section 3]{MR4257071}, in particular we require to relax the assumptions of \cite[Proposition 2]{MR4257071}.
%\end{remark}

%%%%%%%%%%%%%%%%%%%%%%%%%%%%%%%%%%%%%%%%%%%%%%%%%%%%%%%%%%%%
%%%%%%%%%%%%%%%%%%%%%%%%%%%%%%%%%%%%%%%%%%%%%%%%%%%%%%%%%%%%
%%%%%%%%%%%%%%%%%%%%%%%%%%%%%%%%%%%%%%%%%%%%%%%%%%%%%%%%%%%%
%%%%%%%%%%%%%%%%%%%%%%%%%%%%%%%%%%%%%%%%%%%%%%%%%%%%%%%%%%%%

\subsection{Regular Solutions}

\EO{In this section, we follow \cite{MR3936891} and introduce the concept of \emph{regular solutions} for the Navier--Stokes equations.}

\begin{definition}[Regular Solution]\label{def:reg_sol}
\EO{Let $(\boldsymbol\Phi,\zeta) \in \mathbf{H}_0^1(\rho,\Omega)\times L^2(\rho,\Omega)/\mathbb{R} $ be a weak solution to \eqref{def:state_eq} associated to a control $\mathcal{U} = ( \mathbf{u}_1,\dots \mathbf{u}_{\ell} ) \in \mathbb{U}_{ad}$. We say that the velocity field $\boldsymbol\Phi$ is regular if for every $\mathbf{g} \in \mathbf{H}^{-1}(\Omega)$ the weak problem: Find $(\boldsymbol\theta,\xi)\in \mathbf{H}_0^1(\Omega)\times L^2(\Omega)/\mathbb{R}$ such that
\begin{equation}
\label{eq:lin_NS}
\begin{aligned}
\displaystyle\int_{\Omega} \left[ \left(\nu \nabla \boldsymbol\theta
-\boldsymbol\Phi\otimes\boldsymbol\theta - \boldsymbol\theta \otimes \boldsymbol\Phi \right): \nabla \mathbf{w}
- 
\xi\textnormal{div }\mathbf{w}\right] 
 & = \langle\mathbf{g},\mathbf{w}\rangle_{\mathbf{H}^{-1}(\Omega),\mathbf{H}_{0}^{1}(\Omega)},
 \\
 \int_{\Omega} s\textnormal{div }\boldsymbol\theta & = 0,
 \end{aligned}
\end{equation}
for all $(\mathbf{w},s) \in \mathbf{H}_{0}^{1}(\Omega) \times L^{2}(\Omega)/\mathbb{R}$, is well--posed.}
\end{definition}

\EO{Let us introduce the linear map
\begin{align}\label{eq:T}
\begin{split}
T: \mathbf{V}(\Omega) \times L^{2}(\Omega)/\mathbb{R} & \rightarrow \mathbf{H}^{-1}(\Omega), \\
(\boldsymbol\theta,\xi) & \mapsto -\nu \Delta \boldsymbol\theta + \text{div}\left(\boldsymbol\Phi \otimes \boldsymbol\theta ) + \text{div}(\boldsymbol\theta \otimes \boldsymbol\Phi \right)
%(\boldsymbol\Phi \cdot \nabla)\boldsymbol\theta + (\boldsymbol\theta \cdot \nabla)\boldsymbol\Phi 
+ \nabla \xi,
\end{split}
\end{align}
where $\mathbf{V}(\Omega) = \{\mathbf{w} \in \mathbf{H}_{0}^{1}(\Omega): \text{div } \mathbf{w} = 0 \text{ in } \Omega\}$. We notice that, as a consequence of Definition \ref{def:reg_sol}, if the velocity field $\boldsymbol\Phi$ is regular, then the map $T$ is an isomorphism from $\mathbf{V}(\Omega) \times L^{2}(\Omega)/\mathbb{R}$ into $\mathbf{H}^{-1}(\Omega)$ \cite[Section 2]{MR3936891}.}
%\end{remark}

\EO{In the following result, we show a well-posedness result that is crucial for the upcoming analysis.}

\begin{theorem}[Well-Posedness in Weighted Spaces]\label{thm:wellposed_lin_NS_weighted}
\EO{Let $(\boldsymbol\Phi,\zeta)$ be a solution to \eqref{def:state_eq} associated to $\mathcal{U} \in \mathbb{U}_{ad}$ such that $\boldsymbol\Phi$ is regular. Then, for every $\mathbf{g} \in \mathbf{H}_{0}^{1}(\rho^{-1},\Omega)'$ the problem: Find $(\boldsymbol\theta,\xi) \in \mathbf{H}_0^1(\rho,\Omega)\times L^2(\rho,\Omega)/\mathbb{R}$ such that
\begin{multline}
\label{eq:lin_NS_weighted}
%\begin{aligned}
\displaystyle\int_{\Omega} \left[ \left(\nu \nabla \boldsymbol\theta
-\boldsymbol\Phi\otimes\boldsymbol\theta - \boldsymbol\theta \otimes \boldsymbol\Phi \right): \nabla \mathbf{w}
- 
\xi\textnormal{div }\mathbf{w}\right]
\\
  = \langle\mathbf{g},\mathbf{w}\rangle_{\mathbf{H}_{0}^{1}(\rho^{-1},\Omega)',\mathbf{H}_{0}^{1}(\rho^{-1},\Omega)}, 
\quad
 \int_{\Omega} s\textnormal{div }\boldsymbol\theta = 0,
% \end{aligned}
\end{multline}
for all $(\mathbf{w},s) \in \mathbf{H}_{0}^{1}(\rho^{-1},\Omega) \times L^{2}(\rho^{-1},\Omega)/\mathbb{R}$, admits a unique solution. In addition, we have the stability bound
\begin{equation}
\label{eq:stability_lin_NS_weighted}
\|\nabla \boldsymbol \theta\|_{\mathbf{L}^{2}(\rho,\Omega)} + \|\xi\|_{L^{2}(\rho,\Omega)}
\lesssim (1 + \| \nabla\boldsymbol\Phi\|_{\mathbf{L}^{p}(\Omega)})
\|\mathbf{g} \|_{\mathbf{H}_0^1(\rho^{-1},\Omega)'},
\end{equation}
where $p \in (4/3 - \epsilon, 2)$ and $\epsilon = \epsilon(\Omega) > 0$.}
\end{theorem}
\begin{proof}
\EO{We adapt the duality argument elaborated in the proof of \cite[Theorem 2.9]{MR3936891} to our weighted setting. To accomplish this task, we introduce the map
\begin{equation}\label{eq:T_omega}
\begin{split}
T_{\rho}: \mathbf{V}(\rho,\Omega) \times L^{2}(\rho,\Omega)/\mathbb{R} & \rightarrow \mathbf{H}_{0}^{1}(\rho^{-1},\Omega)', 
\\
(\boldsymbol\theta,\xi) & \mapsto -\nu \Delta \boldsymbol\theta + \text{div}\left(\boldsymbol\Phi \otimes \boldsymbol\theta ) + \text{div}(\boldsymbol\theta \otimes \boldsymbol\Phi \right) + \nabla \xi,
\end{split}
\end{equation}
where $\mathbf{V}(\rho,\Omega) = \{\mathbf{v} \in \mathbf{H}_{0}^{1}(\rho,\Omega): \textnormal{div } \mathbf{v} = 0  \text{ in } \Omega \}$, and prove that $T_{\rho}$ is an isomorphism on the basis of three steps.}

\EO{\emph{Step 1.} \emph{Well-posedness of the adjoint problem in} $\mathbf{H}_0^1(\Omega) \times L^2(\Omega)/\mathbb{R}$. Given $\boldsymbol\psi \in \mathbf{H}^{-1}(\Omega)$, we introduce the adjoint problem: Find $(\mathbf{z},r)$ such that
\begin{equation}\label{eq:adj_sys}
-\nu \Delta \mathbf{z} - (\boldsymbol\Phi \cdot \nabla)\mathbf{z} 
+ (\nabla \boldsymbol\Phi)^{\intercal}\mathbf{z} + \nabla r 
= 
\boldsymbol\psi \ \text{in } \Omega, 
\,\,
\text{div } \mathbf{z} = 0 \ \text{in } \Omega, 
\,\,
 \mathbf{z} = \boldsymbol{0} \ \text{on } \partial \Omega.
\end{equation}
We also introduce a suitable linear map associated to the system \eqref{eq:adj_sys}:
\begin{align}\label{eq:S}
\begin{split}
S: \mathbf{V}(\Omega) \times L^{2}(\Omega)/\mathbb{R} & \rightarrow \mathbf{H}^{-1}(\Omega), \\
(\mathbf{z},r) & \mapsto -\nu \Delta \mathbf{z} - (\boldsymbol\Phi \cdot \nabla)\mathbf{z} + (\nabla \boldsymbol\Phi)^{\intercal}\mathbf{z} + \nabla r.
\end{split}
\end{align}

In what follows, we prove that $S$ is an isomorphism. As a first step, we derive the bound
\begin{equation}\label{eq:step1_estim}
\| \nabla \mathbf{z}\|_{\mathbf{L}^{2}(\Omega)} + \|r\|_{L^{2}(\Omega)} \lesssim 
\|S(\mathbf{z},r)\|_{\mathbf{H}^{-1}(\Omega)} \quad \forall (\mathbf{z},r) \in \mathbf{V}(\Omega) \times L^{2}(\Omega)/\mathbb{R}.
\end{equation}
Let $\mathbf{g} \in \mathbf{H}^{-1}(\Omega)$ and let $(\mathbf{z},r) \in \mathbf{V}(\Omega) \times L^2(\Omega)/ \mathbb{R}$. Since $\boldsymbol\Phi$ is regular, the map $T$, defined in \eqref{eq:T}, is an isomorphism. Consequently, there exists $(\boldsymbol\theta,\xi) \in \mathbf{V}(\Omega) \times L^{2}(\Omega)/\mathbb{R}$ such that $T(\boldsymbol\theta,\xi) = \mathbf{g}$. Invoke the definitions of $T$ and $S$, given by \eqref{eq:T} and \eqref{eq:S}, respectively, integration by parts, and the fact that $\textrm{div }\boldsymbol\theta = \textrm{div }\boldsymbol\Phi = \textrm{div }\mathbf{z} = 0$ to arrive at
\begin{multline*}
\langle \mathbf{g},\mathbf{z} \rangle = \langle T(\boldsymbol\theta,\xi),\mathbf{z} \rangle = \langle -\nu \Delta \boldsymbol\theta + \text{div}(\boldsymbol\Phi \otimes \boldsymbol\theta) + \text{div}(\boldsymbol\theta \otimes \boldsymbol\Phi) + \nabla \xi , \mathbf{z} \rangle 
\\
= \langle -\nu \Delta \mathbf{z} +  (\nabla \boldsymbol\Phi)^{\intercal} \mathbf{z}  - (\boldsymbol\Phi \cdot \nabla)\mathbf{z} + \nabla r, \boldsymbol\theta \rangle 
= \langle S(\mathbf{z},r), \boldsymbol\theta\rangle 
\\
\lesssim 
\|S(\mathbf{z},r)\|_{\mathbf{H}^{-1}(\Omega)}\|\mathbf{g}\|_{\mathbf{H}^{-1}(\Omega)},
\end{multline*}
where, in the last step, we have utilized the bound $\| \nabla \boldsymbol\theta \|_{\mathbf{L}^2(\Omega)} \lesssim \|\mathbf{g}\|_{\mathbf{H}^{-1}(\Omega)}$; the latter follows from the fact that $\boldsymbol\Phi$ is regular (cf. Definition \ref{def:reg_sol}). Since $\mathbf{g}$ and $(\mathbf{z},r)$ are arbitrary, we can thus conclude that, for every $(\mathbf{z},r) \in \mathbf{V}(\Omega) \times L^2(\Omega)/ \mathbb{R}$, we have
\begin{equation}\label{eq:estim_phi}
\| \nabla \mathbf{z}\|_{\mathbf{L}^{2}(\Omega)} \lesssim \| S(\mathbf{z},r) \|_{\mathbf{H}^{-1}(\Omega)}.
\end{equation}
It remains to bound $\|r\|_{L^{2}(\Omega)}$. Invoke the definition of $S$ given in \eqref{eq:S} to obtain
\begin{multline}\label{eq:estim_pi}
\|\nabla r\|_{\mathbf{H}^{-1}(\Omega)} \leq \|S(\mathbf{z},r)\|_{\mathbf{H}^{-1}(\Omega)} + \|\nu \Delta \mathbf{z}\|_{\mathbf{H}^{-1}(\Omega)} \\
+ \|(\boldsymbol\Phi \cdot \nabla)\mathbf{z}\|_{\mathbf{H}^{-1}(\Omega)} + \|(\nabla \boldsymbol\Phi)^{\intercal}\mathbf{z}\|_{\mathbf{H}^{-1}(\Omega)}.
\end{multline}
It is clear that $\| \Delta \mathbf{z}\|_{\mathbf{H}^{-1}(\Omega)} \leq \|\nabla \mathbf{z}\|_{\mathbf{L}^{2}(\Omega)}$. To bound the first convective term on the right-hand side of \eqref{eq:estim_pi}, we invoke H\"older's inequality, the standard Sobolev embedding $\mathbf{H}_{0}^{1}(\Omega) \hookrightarrow \mathbf{L}^{\beta}(\Omega)$, which holds for every $\beta < \infty$, and the first embedding result of Theorem \ref{thm:embedding_result}. These arguments reveal that
\[
\| (\boldsymbol\Phi \cdot \nabla)\mathbf{z} \|_{\mathbf{H}^{-1}(\Omega)} 
 \lesssim \|\boldsymbol\Phi\|_{\mathbf{L}^{2+\varepsilon}(\Omega)}\|\nabla \mathbf{z}\|_{\mathbf{L}^{2}(\Omega)}
\lesssim 
 \| \nabla \boldsymbol\Phi\|_{\mathbf{L}^{2}(\rho,\Omega)}\|\nabla \mathbf{z}\|_{\mathbf{L}^{2}(\Omega)}.
\]
Here, $\varepsilon >0$ is as in the statement of Theorem \ref{thm:embedding_result}. The second convective term on the right-hand side of \eqref{eq:estim_pi} can be controlled as follows:
\begin{multline}\label{eq:bound_conv_term_H-1}
\| (\nabla\boldsymbol\Phi)^{\intercal}\mathbf{z} \|_{\mathbf{H}^{-1}(\Omega)} 
 \leq 
\sup_{\mathbf{v} \in \mathbf{H}_{0}^{1}(\Omega)}\dfrac{\|\nabla\boldsymbol\Phi\|_{\mathbf{L}^{2}(\rho,\Omega)}\|\mathbf{z}\|_{\mathbf{L}^{\kappa}(\Omega)}\|\mathbf{v}\|_{\mathbf{L}^{\mu}(\Omega)}\|\rho^{-\frac{1}{2}}\|_{\mathbf{L}^{\varsigma}(\Omega)} }{\|\nabla \mathbf{v}\|_{\mathbf{L}^{2}(\Omega)}} \\
\lesssim
\|\nabla\boldsymbol\Phi\|_{\mathbf{L}^{2}(\rho,\Omega)}\|\nabla\mathbf{z}\|_{\mathbf{L}^{2}(\Omega)},
\qquad
\kappa^{-1} + \mu^{-1} + \varsigma^{-1} = 2^{-1},
\end{multline}
upon utilizing $\mathbf{H}_{0}^{1}(\Omega) \hookrightarrow \mathbf{L}^{\beta}(\Omega)$ ($\beta < \infty$). To bound $\|\rho^{-\frac{1}{2}}\|_{\mathbf{L}^{\varsigma}(\Omega)}$ we invoke
%%the reverse H\"older's inequality of 
Proposition \ref{pro:reverse}
%Theorem \ref{thm:embedding_result}, 
and the fact that $\varsigma$ can be written as $\varsigma = 2 + \delta$ for $\delta > 0$ arbitrarily small. In fact, we have
\begin{equation}\label{eq:bound_rho_-1/2}
\| \rho^{-\frac{1}{2}} \|^{\varsigma}_{\mathbf{L}^{\varsigma}(\Omega)}
=
\left( \int_{\Omega} \rho^{-1 - \frac{\delta }{2} } \right)
\lesssim |\Omega|^{-\frac{\delta }{2}} \left( \int_{\Omega} \rho^{-1} \right)^{1+\frac{\delta }{2}}.
\end{equation}
%In the latter inequality we have utilized the reverse H\"older's inequality of Proposition \ref{pro:reverse} and the fact that $\varsigma$ can be written as $\varsigma = 2 + \delta$, for $\delta > 0$ arbitrarily small.
Replace \eqref{eq:estim_phi} and the estimates previously obtained into \eqref{eq:estim_pi} to obtain
\begin{equation*}
\begin{aligned}
\|\nabla r\|_{\mathbf{H}^{-1}(\Omega)}
& \lesssim
\|S(\mathbf{z},r)\|_{\mathbf{H}^{-1}(\Omega)} + 
\left(\nu 
+
 \|\nabla \boldsymbol\Phi\|_{\mathbf{L}^{2}(\rho,\Omega)} 
 \right)\|\nabla \mathbf{z}\|_{\mathbf{L}^{2}(\Omega)}
\\
& \lesssim
\left(1 + \nu 
+
 \|\nabla \boldsymbol\Phi\|_{\mathbf{L}^{2}(\rho,\Omega)} 
 \right)
\|S(\mathbf{z},r)\|_{\mathbf{H}^{-1}(\Omega)}. 
\end{aligned}
\end{equation*}
This bound, together with \eqref{eq:estim_phi}, allows us to obtain \eqref{eq:step1_estim}. With estimate  \eqref{eq:step1_estim} at hand, we can thus deduce that the linear and bounded operator $S$ is injective with a closed range in $\mathbf{H}^{-1}(\Omega)$. The surjectivity of $S$ can be obtained with similar arguments to the ones developed in the \emph{Step 1} of the proof of \cite[Theorem 2.9]{MR3936891}.}
%From the step 1 of the proof of \cite[Theorem 2.9]{MR3162396} it follows that $S$ is an isomorphism.

\EO{\emph{Step 2}. \emph{Well-posedness of the adjoint problem in $\mathbf{H}_{0}^{1}(\rho^{-1},\Omega) \times L^{2}(\rho^{-1},\Omega)/\mathbb{R}$}. The purpose of this step is to prove that problem \eqref{eq:adj_sys} is well-posed in the space $\mathbf{H}_{0}^{1}(\rho^{-1},\Omega) \times L^{2}(\rho^{-1},\Omega)/\mathbb{R} $ whenever $\boldsymbol \psi \in \mathbf{H}_0^1(\rho,\Omega)'$. To accomplish this task, we introduce the map
\begin{align}\label{eq:Sp}
\begin{split}
S_{\rho}: \mathbf{V}(\rho^{-1},\Omega) \times L^{2}(\rho^{-1},\Omega)/\mathbb{R} & \rightarrow \mathbf{H}_0^1(\rho,\Omega)', \\
(\mathbf{z},r) & \mapsto -\nu \Delta \mathbf{z} - (\boldsymbol\Phi \cdot \nabla)\mathbf{z} + (\nabla \boldsymbol\Phi)^{\intercal}\mathbf{z} + \nabla r.
\end{split}
\end{align}
In what follows, we prove that the linear map $S_{\rho}$ is an isomorphism.

\emph{Step 2.1}. $S_{\rho}$ \emph{is surjective:} Let $\boldsymbol\psi \in \mathbf{H}_{0}^{1}(\rho,\Omega)'$. Since $ \mathbf{H}_{0}^{1}(\rho,\Omega)' \subset \mathbf{H}^{-1}(\Omega)$ (cf. Lemma \ref{lemma:density_H-1_dualH01_omega-1}), we immediately deduce that $\boldsymbol\psi \in \mathbf{H}^{-1}(\Omega)$. As a consequence, the well-posedness results obtained in \emph{Step 1} yield the existence of a unique solution $(\mathbf{z},r) \in \mathbf{H}_0^1(\Omega) \times L^{2}(\Omega)/\mathbb{R}$ to system \eqref{eq:adj_sys} together with a suitable stability bound. We now prove that $(\mathbf{z},r)$ belongs to $\mathbf{H}_{0}^{1}(\rho^{-1},\Omega) \times L^{2}(\rho^{-1},\Omega)/\mathbb{R}$. To accomplish this task, we first observe that $(\mathbf{z},r)$ can be seen as the solution to the Stokes problem
\begin{equation*}
-\nu \Delta \mathbf{z} + \nabla r = \boldsymbol\psi+ (\boldsymbol\Phi \cdot \nabla)\mathbf{z} - (\nabla \boldsymbol\Phi)^{\intercal}\mathbf{z} \ \text{in } \Omega, 
\quad
 \text{div } \mathbf{z} = 0 \ \text{in } \Omega, \quad \mathbf{z} = \boldsymbol{0} \ \text{on } \partial \Omega,
\end{equation*}
and notice that the forcing term of the momentum equation belongs to $\mathbf{H}_{0}^{1}(\rho,\Omega)'$. In fact, the control of the convective term $(\boldsymbol\Phi \cdot \nabla)\mathbf{z}$ in $\mathbf{H}_{0}^{1}(\rho,\Omega)'$ is as follows:
\begin{equation*}
\begin{aligned}
\|(\boldsymbol\Phi \cdot \nabla)\mathbf{z} \|_{\mathbf{H}_{0}^{1}(\rho,\Omega)'} 
& 
\leq 
\sup_{\mathbf{v} \in \mathbf{H}_{0}^{1}(\rho,\Omega)}
\dfrac{
\|\boldsymbol\Phi\|_{\mathbf{L}^{\mu}(\Omega)}\|\nabla\mathbf{z}\|_{\mathbf{L}^{2}(\Omega)}\|\mathbf{v}\|_{\mathbf{L}^{\varsigma}(\Omega)}}{\|\nabla \mathbf{v}\|_{\mathbf{L}^{2}(\rho,\Omega)}} \\
& 
\lesssim 
\| \nabla \boldsymbol\Phi\|_{\mathbf{L}^{p}(\Omega)}
\|\nabla\mathbf{z}\|_{\mathbf{L}^{2}(\Omega)},
\quad
\mu^{-1} + \varsigma^{-1} = 2^{-1},
\end{aligned}
\end{equation*}
upon setting $\varsigma = 2 +\varepsilon$ with $\varepsilon$ being dictated by Theorem \ref{thm:embedding_result}. Notice that, we have also utilized the fact that $\boldsymbol\Phi \in \mathbf{W}_0^{1,p}(\Omega)$ for every $p \in (4/3 - \epsilon, 2)$, where $\epsilon = \epsilon(\Omega) > 0$; see estimate \eqref{eq:NS_p_stab}. The second convective term can be bounded in light of similar arguments:
\begin{equation*}
\begin{aligned}
\|(\nabla\boldsymbol\Phi)^{\intercal} \mathbf{z} \|_{\mathbf{H}_{0}^{1}(\rho,\Omega)'} 
&
\leq 
\sup_{\mathbf{v} \in \mathbf{H}_{0}^{1}(\rho,\Omega)}
\dfrac{
\| \nabla\boldsymbol\Phi\|_{\mathbf{L}^{p}(\Omega)} \| \mathbf{z} \|_{\mathbf{L}^{\mu}(\Omega)} \| \mathbf{v} \|_{\mathbf{L}^{\varsigma}(\Omega)}
}
{\|\nabla \mathbf{v}\|_{\mathbf{L}^{2}(\rho,\Omega)}} 
\\
& \lesssim 
\| \nabla\boldsymbol\Phi\|_{\mathbf{L}^{p}(\Omega)} 
\|\nabla\mathbf{z}\|_{\mathbf{L}^{2}(\Omega)},
\quad
p^{-1}
+ \mu^{-1} + \varsigma^{-1} = 1,
\end{aligned}
\end{equation*}
upon setting, again, $\varsigma = 2 +\varepsilon$ with $\varepsilon$ being dictated by Theorem \ref{thm:embedding_result}. Notice that we have also utilized the standard Sobolev embedding $\mathbf{H}_{0}^{1}(\Omega) \hookrightarrow \mathbf{L}^{\beta}(\Omega)$, which holds for every $\beta < \infty$. Having proved that the forcing term of the momentum equation belongs to $\mathbf{H}_{0}^{1}(\rho,\Omega)'$, we invoke \cite[Theorem 17]{MR3906341} to conclude that $(\mathbf{z},r) \in \mathbf{V}(\rho^{-1},\Omega) \times L^{2}(\rho ^{-1},\Omega)/\mathbb{R}$ together with the bound
\begin{equation}
\begin{aligned}
\|\nabla \mathbf{z}\|_{\mathbf{L}^{2}(\rho^{-1},\Omega)} + \|r\|_{L^{2}(\rho^{-1},\Omega)}
\lesssim
\| \boldsymbol\psi \|_{\mathbf{H}_0^{1}(\rho,\Omega)'} (1 + \| \nabla\boldsymbol\Phi\|_{\mathbf{L}^{p}(\Omega)}),
\end{aligned}
\label{eq:stability_phi_weighted}
\end{equation}
where we have utilized the bounds $\|\nabla\mathbf{z}\|_{\mathbf{L}^{2}(\Omega)} \lesssim \| \boldsymbol\psi \|_{\mathbf{H}^{-1}(\Omega)} \lesssim \| \boldsymbol\psi \|_{\mathbf{H}_0^{1}(\rho,\Omega)'} $. The first bound follows from the results of \emph{Step 1} while the second one follows from the item (ii) in Lemma \ref{lemma:density_H-1_dualH01_omega-1}. We have thus proved that $S_{\rho}$ is surjective.

\emph{Step 2.2}. $S_{\rho}$ \emph{is injective}: Let $(\mathbf{z},r) \in \mathbf{V}(\rho^{-1},\Omega) \times L^{2}(\rho ^{-1},\Omega)/\mathbb{R}$ be such that $S_{\rho}(\mathbf{z},r) = \mathbf{0}$. Since $\mathbf{V}(\rho^{-1},\Omega) \times L^{2}(\rho ^{-1},\Omega)/\mathbb{R} \subset \mathbf{V}(\Omega) \times L^{2}(\Omega)/\mathbb{R}$, we have that $S(\mathbf{z},r) = \mathbf{0}$. The fact that $S$ is an isomorphism allow us to conclude that $(\mathbf{z},r) = (\mathbf{0},0)$.}

\EO{\emph{Step 3}. \emph{Well-posedness of problem \eqref{eq:lin_NS_weighted}.}
We prove that problem \eqref{eq:lin_NS_weighted} is well-posed. To accomplish this task, we proceed on the basis of a density argument. Let $\mathbf{g} \in \mathbf{H}_{0}^{1}(\rho^{-1},\Omega)'$. Since $\mathbf{H}^{-1}(\Omega)$ is dense in $\mathbf{H}_{0}^{1}(\rho^{-1},\Omega)'$ (cf.~Lemma \ref{lemma:density_H-1_dualH01_omega-1}), there exists a sequence $\{\mathbf{g}_{k}\}_{k \in \mathbb{N}} \subset \mathbf{H}^{-1}(\Omega)$ such that $\mathbf{g}_{k} \rightarrow \mathbf{g}$ in $\mathbf{H}_{0}^{1}(\rho^{-1},\Omega)'$ as $k \uparrow \infty$. On the other hand, since $\boldsymbol \Phi$ is regular, the map $T$ is an isomorphism. Consequently, for every $k \in \mathbb{N}$, there exists a unique pair $(\boldsymbol \theta_k,\xi_k) \in \mathbf{H}_0^1(\Omega) \times L^2(\Omega)/\mathbb{R}$ that solves problem \eqref{eq:lin_NS} with $\mathbf{g}$ being replaced by $\mathbf{g}_k$.}

\EO{\emph{Step 3.1}. \emph{$\{ (\boldsymbol \theta_k,\xi_k) \}_{k \in \mathbb{N}}$ is bounded in $\mathbf{H}_{0}^{1}(\rho,\Omega) \times L^{2}(\rho,\Omega)/\mathbb{R}$.} Let $\boldsymbol \psi \in \mathbf{H}_0^1(\rho,\Omega)'$. The results obtained in \emph{Step 2} guarantee that $S_{\rho}$, which is defined in \eqref{eq:Sp}, is an isomorphism. As a consequence, there exists a pair $(\mathbf{z},r) \in \mathbf{V}(\rho^{-1},\Omega) \times L^{2}(\rho ^{-1},\Omega)/\mathbb{R}$ such that $\boldsymbol \psi = S_{\rho}(\mathbf{z},r)$. Let us now observe that
\begin{equation}
\langle \boldsymbol \psi, \boldsymbol \theta_k \rangle 
=
\langle S_{\rho}(\mathbf{z},r),  \boldsymbol\theta_k \rangle 
=
\langle \mathbf{z}, T(\boldsymbol \theta_k,\xi_k)\rangle 
=
\langle  \mathbf{g}_k, \mathbf{z} \rangle. 
\label{eq:duality_trick}
\end{equation}
With the previous identity at hand, the stability bound \eqref{eq:stability_phi_weighted} reveals that
\begin{align*}
|\langle \boldsymbol \psi, \boldsymbol \theta_k \rangle | & \leq
\|\mathbf{g}_k \|_{\mathbf{H}_0^1(\rho^{-1},\Omega)'}
\| \nabla \mathbf{z}\|_{\mathbf{L}^2(\rho^{-1},\Omega)} \\
& \lesssim
\|\mathbf{g}_k \|_{\mathbf{H}_0^1(\rho^{-1},\Omega)'}
\| \boldsymbol \psi \|_{\mathbf{H}_0^1(\rho,\Omega)'}(1 + \|\nabla \boldsymbol{\Phi}\|_{\mathbf{L}^p(\Omega)}).
\end{align*}
The arbitrariness of $\boldsymbol \psi$ allows us to deduce the following bound for $\boldsymbol \theta_k $ and $k \in \mathbb{N}$:
$
\| \nabla \boldsymbol \theta_k  \|_{\mathbf{L}^2(\rho,\Omega)}
\lesssim
\| \mathbf{g}_k \|_{\mathbf{H}_0^1(\rho^{-1},\Omega)'}(1 + \|\nabla \boldsymbol{\Phi}\|_{\mathbf{L}^p(\Omega)}).
$
This estimate and the inf-sup condition on weighted spaces of \cite[Lemma 6.1]{MR4081912} yield the boundedness of the sequence $\{ \| \xi_k \|_{L^2(\rho,\Omega)} \}_{k\in \mathbb{N}}$ in $\mathbb{R}$.

\emph{Step 3.2}. \emph{Existence of solutions for \eqref{eq:lin_NS_weighted}.} Since 
the sequence $\{ (\boldsymbol \theta_k,\xi_k)  \}_{k \in \mathbb{N}}$ is bounded in $\mathbf{H}_{0}^{1}(\rho,\Omega) \times L^{2}(\rho,\Omega)/\mathbb{R}$, we deduce the existence of a nonrelabeled subsequence $\{ (\boldsymbol \theta_k,\xi_k) \}_{k \in \mathbb{N}}$ such that
\[
\boldsymbol \theta_k \rightharpoonup \boldsymbol \theta \textrm{ in } \mathbf{H}_{0}^{1}(\rho,\Omega),
\qquad
\xi_k \rightharpoonup \xi \textrm{ in } L^{2}(\rho,\Omega)/\mathbb{R},
\qquad
k \uparrow \infty.
\]
In what follows, we show that $(\boldsymbol \theta,\xi) \in \mathbf{H}_{0}^{1}(\rho,\Omega) \times L^{2}(\rho,\Omega)/\mathbb{R}$ solves the system \eqref{eq:lin_NS_weighted}. To accomplish this task, we let $(\mathbf{w},s)$ be an arbitrary pair in $\mathbf{H}_{0}^{1}(\rho^{-1},\Omega) \times L^{2}(\rho^{-1},\Omega)/\mathbb{R}$ and observe that $\left| \int_{\Omega} \nabla (\boldsymbol \theta - \boldsymbol \theta_{k}) : \nabla \mathbf{w} \right| \rightarrow 0$ and that
\[
\left| \int_{\Omega} (\boldsymbol \Phi \otimes \boldsymbol \theta - \boldsymbol \Phi \otimes \boldsymbol \theta_{k}) : \nabla \mathbf{w} \right|
\leq 
\| \boldsymbol \Phi   \|_{\mathbf{L}^4(\rho,\Omega)}
\| \boldsymbol \theta - \boldsymbol \theta_{k} \|_{\mathbf{L}^4(\rho,\Omega)}
\| \nabla \mathbf{w} \|_{\mathbf{L}^2(\rho^{-1},\Omega)}
\rightarrow 0
\]
as $k \uparrow \infty$; the second convergence result being a consequence of the weighted compact Sobolev embedding $\mathbf{H}_0^1(\rho,\Omega) \hookrightarrow \mathbf{L}^4(\rho,\Omega)$ \cite[Theorem 4.12]{MR2797702} (see also \cite[Proposition 2]{MR3998864}). Similarly, we have
\[
\left| \int_{\Omega} (\boldsymbol \theta - \boldsymbol \theta_{k}) \otimes \boldsymbol \Phi : \nabla \mathbf{w} \right|,
\quad
\left| \int_{\Omega} (\xi - \xi_k) \textrm{div } \mathbf{w} \right|,
\quad
\left| \int_{\Omega} s \textrm{div } (\boldsymbol \theta - \boldsymbol \theta_{k} )\right| \rightarrow 0
\]
% $\left| \int_{\Omega} (\boldsymbol \theta \otimes \boldsymbol \Phi - \boldsymbol \theta_{k} \otimes \boldsymbol \Phi) : \nabla \mathbf{w} \right| \rightarrow 0$, $\left| \int_{\Omega} (\xi - \xi_k) \textrm{div } \mathbf{w} \right| \rightarrow 0$, 
% and $| \int_{\Omega} s \textrm{div } (\boldsymbol \theta - \boldsymbol \theta_{k} ) | \rightarrow 0$ 
 as $k \uparrow \infty$. We have thus proved that $(\boldsymbol \theta,\xi) \in \mathbf{H}_{0}^{1}(\rho,\Omega) \times L^{2}(\rho,\Omega)/\mathbb{R}$ is a solution to problem \eqref{eq:lin_NS_weighted}.}

\EO{\emph{Step 3.3.} \emph{Stability bound.} Let $\boldsymbol \psi \in \mathbf{H}_0^1(\rho,\Omega)'$ and let $(\mathbf{z},r) \in \mathbf{H}_0^1(\rho^{-1},\Omega) \times L^{2}(\rho ^{-1},\Omega)/\mathbb{R}$ be the unique solution to \eqref{eq:adj_sys}. Similar arguments to the ones utilized to obtain \eqref{eq:duality_trick} combined with the stability bound \eqref{eq:stability_phi_weighted} yield
\begin{multline*}
\langle \boldsymbol \psi, \boldsymbol \theta \rangle 
=
\langle  \mathbf{g}, \mathbf{z} \rangle
\leq 
\|\mathbf{g} \|_{\mathbf{H}_0^1(\rho^{-1},\Omega)'}
\| \nabla \mathbf{z}\|_{\mathbf{L}^2(\rho^{-1},\Omega)}
\\
\lesssim
\|\mathbf{g} \|_{\mathbf{H}_0^1(\rho^{-1},\Omega)'}
\| \boldsymbol\psi \|_{\mathbf{H}_0^{1}(\rho,\Omega)'} (1 + \| \nabla\boldsymbol\Phi\|_{\mathbf{L}^{p}(\Omega)}).
\end{multline*}
Since $\boldsymbol \psi$ is an arbitrary element of $\mathbf{H}_0^1(\rho,\Omega)'$, we can thus deduce that 
\[
\| \nabla \boldsymbol \theta \|_{\mathbf{L}^2(\rho,\Omega)} \lesssim \|\mathbf{g} \|_{\mathbf{H}_0^1(\rho^{-1},\Omega)'}(1 + \| \nabla\boldsymbol\Phi\|_{\mathbf{L}^{p}(\Omega)}).
\]
We now utilize the inf-sup condition on weighted spaces of \cite[Lemma 6.1]{MR4081912} to control the pressure: $\| \xi \|_{L^2(\rho,\Omega)} \lesssim \|\mathbf{g} \|_{\mathbf{H}_0^1(\rho^{-1},\Omega)'}(1 + \| \nabla\boldsymbol\Phi\|_{\mathbf{L}^{p}(\Omega)})$.

\emph{Step 3.4} \emph{The map $T_{\rho}$ is injective}. Let $(\boldsymbol \theta,\xi) \in \mathbf{H}_{0}^{1}(\rho,\Omega) \times L^{2}(\rho,\Omega)/\mathbb{R}$ be such that $T_{\rho}(\boldsymbol \theta,\xi) = \mathbf{0}$. This immediately implies that $\langle T_{\rho}(\boldsymbol \theta,\xi), \mathbf{z} \rangle = 0$ for every $\mathbf{z} \in \mathbf{H}_0^1(\rho^{-1},\Omega)$.  An argument based on integration by parts thus reveals that $\langle \boldsymbol \theta, S_{\rho}(\mathbf{z},r) \rangle = 0$ for every $(\mathbf{z},r) \in \mathbf{H}_{0}^{1}(\rho^{-1},\Omega) \times L^{2}(\rho^{-1},\Omega)/\mathbb{R}$. This implies that $\boldsymbol \theta = \mathbf{0}$.  Since $T_{\rho}(\boldsymbol \theta,\xi) = \mathbf{0}$, we invoke the definition of $T_{\rho}$ to deduce that $\nabla \xi = \mathbf{0}$ and thus that $\xi = 0$. This concludes the proof. \qed}
\end{proof}

\begin{remark}[\EO{Well-Posedness in $\mathbf{W}_{0}^{1,p}(\Omega) \times L^p(\Omega)/ \mathbb{R}$}]
\label{rem:reg_linearized_vel}
\EO{Let $\mathbf{g} \in \mathbf{W}^{-1,p}(\Omega)$ and $p'$ be such that $1/p + 1/p' = 1$. Let us denote by $(\boldsymbol\theta,\xi)$ the weak solution to the system \eqref{eq:lin_NS_weighted} with
\[
\langle\mathbf{g},\mathbf{w}\rangle_{\mathbf{W}^{-1,p}(\Omega),\mathbf{W}_{0}^{1,p'}(\Omega)}
\]
as a forcing term. An adaptation of the arguments elaborated in the proof of Theorem \ref{thm:wellposed_lin_NS_weighted}, that are in turn inspired by the ones in \cite[Theorem 2.9]{MR3936891}, show that problem \ref{eq:lin_NS_weighted} is well posed in $\mathbf{W}_{0}^{1,p}(\Omega) \times L^p(\Omega)/ \mathbb{R}$ whenever $p$ belongs to $(4/3,2)$. This results holds under the assumption that $\Omega$ is Lipschitz and therefore improves on \cite[Theorem 2.9]{MR3936891} where $\partial \Omega \in C^2$. We notice that the only place in the proof of \cite[Theorem 2.9]{MR3936891} where such a regularity on $\Omega$ is needed is \cite[estimate (2.19)]{MR3936891}. Since $p \in (4/3,2)$ and thus $p' \in (2,4)$, \cite[estimate (2.19)]{MR3936891} on Lipschitz domains can be obtained upon utilizing \cite[Theorem 1.6, (1.52)]{MR2763343}.}
\end{remark}
%%%%%%%%%%%%%%%%%%%%%%%%%%%%%%%%%%%%%%%%%%%%%%%%%%%%%%%%%%%%
%%%%%%%%%%%%%%%%%%%%%%%%%%%%%%%%%%%%%%%%%%%%%%%%%%%%%%%%%%%%
%%%%%%%%%%%%%%%%%%%%%%%%%%%%%%%%%%%%%%%%%%%%%%%%%%%%%%%%%%%%
%%%%%%%%%%%%%%%%%%%%%%%%%%%%%%%%%%%%%%%%%%%%%%%%%%%%%%%%%%%%

\subsubsection{\EO{Differentiability Properties of a Solution Operator}}
\label{sec:sol_operator}

\EO{In this section, we investigate differentiability properties for a solution map associated to system \eqref{def:state_eq} around a regular velocity field $\mathbf{y}$. We present some of these properties in  the following theorem.}

\begin{theorem}[Differentiability of $\mathcal{U} \mapsto (\mathbf{y},p)$]\label{thm:diff_prop}
\EO{Let $\bar{\mathcal{U}} \in \mathbb{U}_{ad}$ and let $(\bar{\mathbf{y}},\bar{p}) \in \mathbf{H}_{0}^{1}(\rho,\Omega) \times L^{2}(\rho,\Omega)/\mathbb{R}$ be a solution to \eqref{def:state_eq}. If $\bar{\mathbf{y}}$ is regular, then there exist open neighborhoods $\mathcal{O}(\bar{\mathcal{U}}) \subset [\mathbb{R}^{2}]^{\ell} $, $\mathcal{O}(\bar{\mathbf{y}}) \subset \mathbf{V}(\rho,\Omega) $, and $\mathcal{O}(\bar{p}) \subset L^{2}(\rho,\Omega)/\mathbb{R}$ of $ \bar{\mathcal{U}}$, $\bar{\mathbf{y}}$, and $\bar{p}$, respectively, and a map of class $C^2$,
\begin{equation}\label{eq:sol_op_Q}
Q: \mathcal{O}(\bar{\mathcal{U}}) \to \mathcal{O}(\bar{\mathbf{y}}) \times \mathcal{O}(\bar{p}),
\end{equation}
such that $Q(\bar{\mathcal{U}}) = (\bar{\mathbf{y}},\bar{p})$. In addition, the neighborhood $\mathcal{O}(\bar{\mathcal{U}})$ can be taken such that, for each $\mathcal{U} \in \mathcal{O}(\bar{\mathcal{U}})$,
\begin{itemize}
\item[\textnormal{(i)}] the pair $(\mathbf{y}_{\mathcal{U}},p_{\mathcal{U}}) = Q(\mathcal{U})$ uniquely solves \eqref{def:state_eq} in $\mathcal{O}(\bar{\mathbf{y}}) \times \mathcal{O}(\bar{p})$,
\item[\textnormal{(ii)}] the map $Q'(\mathcal{U}):[\mathbb{R}^{2}]^{\ell} \rightarrow \mathbf{V}(\rho,\Omega) \times L^{2}(\rho,\Omega)/\mathbb{R}$ is an isomorphism,
\item[\textnormal{(iii)}] if $\mathcal{V} \in [\mathbb{R}^{2}]^{\ell}$, then $(\boldsymbol\theta,\xi) := Q'(\mathcal{U})\mathcal{V}
%% \in \mathbf{V}(\rho,\Omega) \times L^{2}(\rho,\Omega)/\mathbb{R}
$ corresponds to the unique solution to \eqref{eq:lin_NS_weighted} with $\boldsymbol\Phi$ and $\mathbf{g}$ being replaced by $\mathbf{y}_{\mathcal{U}}$ and $\sum_{t \in \mathcal{D}}\mathbf{v}_{t}\delta_{t}$, respectively, and
\item[\textnormal{(iv)}] if $\mathcal{V}_{1},\mathcal{V}_{2} \in [\mathbb{R}^{2}]^{\ell}$, then $(\boldsymbol{\psi},\gamma):=Q''(\mathcal{U})\mathcal{V}_{1}\mathcal{V}_{2}$ corresponds to the unique solution to
\begin{multline}\label{eq:second_der_sol_map_S}
\int_{\Omega} ( [\nu \nabla \boldsymbol{\psi} - 
\mathbf{y}_{\mathcal{U}} \otimes \boldsymbol{\psi}
-
\boldsymbol{\psi} \otimes \mathbf{y}_{\mathcal{U}}]: \nabla\mathbf{v}
- 
\gamma\textnormal{div }\mathbf{v}) \\
 =    \int_{\Omega}(\boldsymbol{\theta}_{\mathcal{V}_{1}} \otimes \boldsymbol{\theta}_{\mathcal{V}_{2}} + \boldsymbol{\theta}_{\mathcal{V}_{2}} \otimes\boldsymbol{\theta}_{\mathcal{V}_{1}}) : \nabla \mathbf{v},\qquad 
\int_{\Omega} q\textnormal{div }\boldsymbol{\psi}  = 0,
\end{multline}
for all $(\mathbf{v},q)\in \mathbf{H}_{0}^{1}(\rho^{-1},\Omega) \times L^{2}(\rho^{-1},\Omega)/\mathbb{R}$. Here, $(\boldsymbol \theta_{\mathcal{V}_{i}},\xi_{\mathcal{V}_{i}})=Q'(\mathcal{U})\mathcal{V}_{i}$, with $i\in\{1,2\}$, corresponds to the unique solution to \eqref{eq:lin_NS_weighted} with $\boldsymbol\Phi$ and $\mathbf{g}$ being replaced by $\mathbf{y}_{\mathcal{U}}$ and $\sum_{t \in \mathcal{D}}\mathbf{v}_{t}\delta_{t}$, respectively.
\end{itemize}
}
\end{theorem}
\begin{proof}
\EO{The proof follows from slight modifications of the arguments elaborated in the proof of \cite[Theorem 2.10 and Corollary 2.11]{MR3936891} upon utilizing the results of Theorem \ref{thm:wellposed_lin_NS_weighted}. For brevity, we skip the details.
\qed}
\end{proof}

\EO{We conclude this section with the following Lipschitz property for $Q$, which will be of importance to study second order conditions in Section \ref{subsec:second_order_optimality_conditions}.
}

\begin{lemma}[Lipschitz Property]\label{lemma:lips_prop}
\EO{In the framework of Theorem \ref{thm:diff_prop}, we have the following Lipschitz property for the map $Q$:}
\begin{equation}
\EO{
\| \nabla( \mathbf{y} - \bar{\mathbf{y}})\|_{\mathbf{L}^{2}(\rho,\Omega)} + \|p - \bar{p} \|_{L^{2}(\rho,\Omega)}
\lesssim
\| \mathcal{U} - \bar{\mathcal{U}} \|_{[\mathbb{R}^{2}]^{\ell}}
\quad \forall \mathcal{U} \in \mathcal{O}(\bar{\mathcal{U}}),
}
\label{eq:LIpschitz_property}
\end{equation}
\EO{with a hidden constant that depends on $Q'$ and $\mathcal{O}(\bar{\mathcal{U}})$.}
\end{lemma}
\begin{proof}
\EO{In view of the results of Theorem \ref{thm:diff_prop}, we can choose an open, bounded, and convex neighborhood $\mathcal{O}(\bar{\mathcal{U}})$ of $\bar{\mathcal{U}}$ such that $Q'(\mathcal{U}):[\mathbb{R}^{2}]^{\ell} \rightarrow \mathbf{V}(\rho,\Omega) \times L^{2}(\rho,\Omega)/\mathbb{R}$ is an isomorphism and $\|Q'(\mathcal{U})\| \leq \mathcal{M}_{Q}$ for every $\mathcal{U} \in \mathcal{O}(\bar{\mathcal{U}})$. Here, $\mathcal{M}_{Q} > 0$ and $\| \cdot \|$ denotes the norm in the space of linear and continuous operators from $[\mathbb{R}^{2}]^{\ell}$ into $\mathbf{V}(\rho,\Omega) \times L^{2}(\rho,\Omega)/\mathbb{R}$. Thus, as a consequence of the mean value theorem for operators \cite[Proposition 5.3.11]{MR2153422}, we have}
\[ 
\EO{\| Q(\mathcal{U}) - Q(\bar{\mathcal{U}}) \|_{\mathbf{H}_{0}^{1}(\rho,\Omega) \times L^{2}(\rho,\Omega)/\mathbb{R}}
\leq
\sup_{\mathfrak{t} \in [0,1]}\|Q'((1-\mathfrak{t})\mathcal{U} + \mathfrak{t}\bar{\mathcal{U}})\|\| \mathcal{U} - \bar{\mathcal{U}} \|_{[\mathbb{R}^{2}]^{\ell}}} 
\]
\end{proof}
\EO{for every $\mathcal{U} \in \mathcal{O}(\bar{\mathcal{U}})$. Invoke the fact that $\|Q'(\mathcal{U})\| \leq \mathcal{M}_{Q}$, for every $\mathcal{U} \in \mathcal{O}(\bar{\mathcal{U}})$, to immediately arrive at the desired bound.
%The proof concludes by estimating the supremum with the bound $M_{Q}$.
\qed}

%%%%%%%%%%%%%%%%%%%%%%%%%%%%%%%%%%%%%%%%%%%%%%%%%%%%%%%%%%%%
%%%%%%%%%%%%%%%%%%%%%%%%%%%%%%%%%%%%%%%%%%%%%%%%%%%%%%%%%%%%
%%%%%%%%%%%%%%%%%%%%%%%%%%%%%%%%%%%%%%%%%%%%%%%%%%%%%%%%%%%%
%%%%%%%%%%%%%%%%%%%%%%%%%%%%%%%%%%%%%%%%%%%%%%%%%%%%%%%%%%%%
%%%%%%%%%%%%%%%%%%%%%%%%%%%%%%%%%%%%%%%%%%%%%%%%%%%%%%%%%%%%
%%%%%%%%%%%%%%%%%%%%%%%%%%%%%%%%%%%%%%%%%%%%%%%%%%%%%%%%%%%%
%%%%%%%%%%%%%%%%%%%%%%%%%%%%%%%%%%%%%%%%%%%%%%%%%%%%%%%%%%%%
%%%%%%%%%%%%%%%%%%%%%%%%%%%%%%%%%%%%%%%%%%%%%%%%%%%%%%%%%%%%
\section{The Optimal Control Problem}\label{sec:ocp}
In this section, we propose and analyze the following weak formulation for the optimal control problem \eqref{def:cost_functional}--\eqref{def:box_constraints}: Find
\begin{equation}\label{eq:weak_cost}
\min \{J(\mathbf{y},\mathcal{U}): \EO{(\mathbf{y},\mathcal{U})\in \mathbf{H}_{0}^{1}(\rho,\Omega)}\times\mathbb{U}_{ad}\},
\end{equation}
subject to the weak formulation of the stationary Navier--Stokes equations
\begin{equation}
\label{eq:weak_st_eq}
\int_{\Omega} \left(\nu \nabla \mathbf{y} : \nabla \mathbf{v} - \mathbf{y}\otimes\mathbf{y}: \nabla \mathbf{v} - p\text{div }\mathbf{v}\right) = \sum_{t\in\mathcal{D}}\langle \mathbf{u}_t\delta_{t},\mathbf{v}\rangle, \quad \int_{\Omega} q\text{div }\mathbf{y} = 0,
\end{equation}
for all $(\mathbf{v},q)\in \mathbf{H}_{0}^{1}(\rho^{-1},\Omega) \times L^{2}(\rho^{-1},\Omega)/\mathbb{R}$. The weight $\rho$ is defined as in \eqref{def:weight_rho}, \EO{where the parameter $\alpha$ belongs to $(0,2)$.} We comment that, since the velocity component $\mathbf{y}$ of a solution to the state equation is sought in $\mathbf{H}_0^1(\rho,\Omega)$, an application of Theorem \ref{thm:embedding_result} guarantees that $\mathbf{y} \in \mathbf{L}^2(\Omega)$. Consequently, all the terms involved in the definition of the cost functional $J$ are well defined.

%%%%%%%%%%%%%%%%%%%%%%%%%%%%%%%%%%%%%%%%%%%%%%%%%%%%%%%%%%%%
%%%%%%%%%%%%%%%%%%%%%%%%%%%%%%%%%%%%%%%%%%%%%%%%%%%%%%%%%%%%
%%%%%%%%%%%%%%%%%%%%%%%%%%%%%%%%%%%%%%%%%%%%%%%%%%%%%%%%%%%%
%%%%%%%%%%%%%%%%%%%%%%%%%%%%%%%%%%%%%%%%%%%%%%%%%%%%%%%%%%%%
\subsection{Existence of Optimal Solutions} 
The existence of an optimal solution \EO{$(\bar{\mathbf{y}},\bar{\mathcal{U}})$} is as follows.
\begin{theorem}[Existence]
The control problem \eqref{eq:weak_cost}--\eqref{eq:weak_st_eq} admits at least one global solution \EO{$(\bar{\mathbf{y}},\bar{\mathcal{U}})\in \mathbf{H}_{0}^{1}(\rho,\Omega)\times\mathbb{U}_{ad}$.}
\label{thm:existence_optimal_solution}
\end{theorem}
\begin{proof}
Let 
%$\{(\mathbf{y}_{k},p_k, \mathcal{U}_{k})\}_{k\in\mathbb{N}}$ 
\EO{$\{(\mathbf{y}_{k}, \mathcal{U}_{k})\}_{k\in\mathbb{N}}$} be a minimizing sequence, i.e., for $k\in \mathbb{N}$, the pair $(\mathbf{y}_{k},  p_k) \in \mathbf{H}_{0}^{1}(\rho,\Omega) \times L^{2}(\rho,\Omega)/\mathbb{R}$
%$(\mathbf{y}_{k},  p_k) \in \mathbf{H}_0^1(\rho,\Omega)\times L^2(\rho,\Omega) / \mathbb{R}$ 
solves
\begin{equation*}
\label{eq:weak_st_eq_k}
\int_{\Omega} \left(\nu \nabla \mathbf{y}_k : \nabla \mathbf{v} - \mathbf{y}_k\otimes\mathbf{y}_k: \nabla \mathbf{v} - p_k\text{div }\mathbf{v}\right) = \sum_{t\in\mathcal{D}}\langle \mathbf{u}_t^k \delta_{t},\mathbf{v}\rangle, \int_{\Omega} q\text{div }\mathbf{y}_k = 0,
\end{equation*}
for all $(\mathbf{v},q)\in\mathbf{H}_{0}^{1}(\rho^{-1},\Omega) \times L^{2}(\rho^{-1},\Omega)/\mathbb{R}$, and 
%the pair $(\mathbf{y}_{k},  p_k)$ together with $\mathcal{U}_k$ are 
\EO{the pair $(\mathbf{y}_{k},\mathcal{U}_{k})$ is} such that
\EO{$J(\mathbf{y}_{k}, \mathcal{U}_{k})\rightarrow \mathfrak{i}:=\inf\{J(\mathbf{y}, \mathcal{U}) : (\mathbf{y}, \mathcal{U})\in \mathbf{H}_{0}^{1}(\rho,\Omega) \times\mathbb{U}_{ad}\}$ as $k\uparrow \infty$.}
%$J(\mathbf{y}_{k}, p_k, \mathcal{U}_{k})\rightarrow \mathfrak{i}:=\inf\{J(\mathbf{y}, p, \mathcal{U}) :(\mathbf{y}, p, \mathcal{U})\in \mathbf{H}_0^1(\rho,\Omega)\times L^2(\rho,\Omega) / \mathbb{R} \times\mathbb{U}_{ad}\}$ as $k\uparrow \infty$. 
Here, for $k \in \mathbb{N}$, we denote $\mathcal{U}_k:= \{  \mathbf{u}_t^k \}_{t \in \mathcal{D}}$. We notice that the existence of solutions for the previously stated problem follows from the results of Section \ref{sec:NS_singular_rhs}.
%\FL{subsection \ref{sec:NS_singular_rhs}.}

Since $\mathbb{U}_{ad}$ is compact, we immediately conclude the existence of a nonrelabeled subsequence $\{\mathcal{U}_{k}\}_{k\in\mathbb{N}}$ such that $\mathcal{U}_{k}\rightarrow \bar{\mathcal{U}}$ in $[\mathbb{R}^{2}]^{\ell}$ with $\bar{\mathcal{U}}\in\mathbb{U}_{ad}$. On the other hand, in view of the stability bound \eqref{eq:NS_weighted_stab}, we conclude that $\{(\mathbf{y}_{k},p_k)\}_{k\in\mathbb{N}}$ is uniformly bounded in $\mathbf{H}_0^1(\rho,\Omega)\times L^2(\rho,\Omega) / \mathbb{R}$.
Consequently, we deduce the existence of a nonrelabeled subsequence $\{(\mathbf{y}_{k},p_k)\}_{k\in\mathbb{N}}$ such that $(\mathbf{y}_{k},p_k)\rightharpoonup (\bar{\mathbf{y}},\bar{p})$ in $\mathbf{H}_0^1(\rho,\Omega)\times L^2(\rho,\Omega)/\mathbb{R}$
as $k\uparrow \infty$; $(\bar{\mathbf{y}},\bar{p})$ being the natural candidate for an optimal state. The rest of the proof is dedicated to prove that $(\bar{\mathbf{y}},\bar{p})$ solves \eqref{eq:weak_st_eq} with $\mathbf{u}_t$ being replaced by $\bar{\mathbf{u}}_t$ for $t\in \mathcal{D}$.

With the weak convergence $(\mathbf{y}_{k},p_k)\rightharpoonup (\bar{\mathbf{y}},\bar{p})$ in $\mathbf{H}_0^1(\rho,\Omega)\times L^2(\rho,\Omega)/\mathbb{R}$
as $k\uparrow \infty$, at hand, we obtain
\[
\int_{\Omega} \nu \nabla (\mathbf{y}_{k} - \bar{\mathbf{y}}) : \nabla \mathbf{v} \rightarrow \!0, 
~
\int_{\Omega} (p_{k} - \bar{p})\text{div }\mathbf{v} \rightarrow 0, ~ \int_{\Omega}  q\text{div }(\mathbf{y}_{k} - \bar{\mathbf{y}}) \rightarrow 0,
\]
%\FL{\begin{equation*}
%\int_{\Omega} \nu \nabla \mathbf{y}_{k} : \nabla \mathbf{v} \rightarrow \!\int_{\Omega} \nu \nabla \bar{\mathbf{y}} : \nabla \mathbf{v}, 
%\quad 
%\int_{\Omega} p_{k}\text{div }\mathbf{v} \rightarrow \!\int_{\Omega} \bar{p}\text{div }\mathbf{v}, 
%\end{equation*}
%\quad
%\begin{equation*}
%\int_{\Omega}  q\text{div }\mathbf{y}_{k} \rightarrow \!\int_{\Omega} q\text{div }\bar{\mathbf{y}},
%\end{equation*}}
as $k\uparrow \infty$, for every $\mathbf{v} \in \mathbf{H}_0^1(\rho^{-1},\Omega)$ and $q \in L^2(\rho^{-1},\Omega)/\mathbb{R}$. On the other hand, the convergence $\mathcal{U}_{k}\rightarrow \bar{\mathcal{U}}$ in $[\mathbb{R}^{2}]^{\ell}$ yields $\sum_{t\in\mathcal{D}}\langle \mathbf{u}_t^{k}\delta_{t},\mathbf{v}\rangle \rightarrow \sum_{t\in\mathcal{D}}\langle \bar{\mathbf{u}}_t\delta_{t},\mathbf{v}\rangle$ as $k\uparrow \infty$. It thus suffices to analyze the convective term. To accomplish this task, we invoke H\"older's inequality to arrive at
\begin{multline*}
\left|\int_{\Omega} \left(\mathbf{y}_{k}\otimes\mathbf{y}_{k} - \bar{\mathbf{y}}\otimes\bar{\mathbf{y}} \right) : \nabla \mathbf{v} \right| \\
\leq 
\left(\|\mathbf{y}_{k}\|_{\mathbf{L}^4(\rho,\Omega)} + \|\bar{\mathbf{y}}\|_{\mathbf{L}^4(\rho,\Omega)}\right) \|\bar{\mathbf{y}} - \mathbf{y}_{k}\|_{\mathbf{L}^4(\rho,\Omega)}\|\nabla \mathbf{v}\|_{\mathbf{L}^2(\rho^{-1},\Omega)}.
\end{multline*}
The compact embedding $\mathbf{H}_0^1(\rho,\Omega) \hookrightarrow \mathbf{L}^4(\rho,\Omega)$, which follows from \cite[Theorem 4.12]{MR2797702} (see also \cite[Proposition 2]{MR3998864}), combined with $\mathbf{y}_{k}\rightharpoonup \bar{\mathbf{y}}$ in $\mathbf{H}_0^1(\rho,\Omega)$, as $k \uparrow \infty$, allow us to conclude that $(\bar{\mathbf{y}},\bar{p})$ solves \eqref{eq:weak_st_eq} with $\mathbf{u}_t$ being replaced by $\bar{\mathbf{u}}_t$ for $t\in \mathcal{D}$; $\bar{\mathcal{U}} = \{ \bar{\mathbf{u}}_t\}_{t \in \mathcal{D}}$.

To conclude the proof, we must prove the optimality of $\bar{\mathcal{U}}$. Observe that $\mathcal{U}_{k}\rightarrow \bar{\mathcal{U}}$ in $[\mathbb{R}^{2}]^{\ell}$, as $k \uparrow \infty$, and that $\mathbf{y}_{k}\rightarrow \bar{\mathbf{y}}$ in $\mathbf{L}^2(\Omega)$, as $k\uparrow \infty$. The latter follows from
\begin{equation*}
\|\bar{\mathbf{y}} - \mathbf{y}_{k}\|_{\mathbf{L}^2(\Omega)} 
\leq 
\|\bar{\mathbf{y}} - \mathbf{y}_{k}\|_{\mathbf{L}^4(\rho,\Omega)} \left(\int_\Omega\rho^{-1} \right)^{\frac{1}{4}}
\lesssim
\|\bar{\mathbf{y}} - \mathbf{y}_{k}\|_{\mathbf{L}^4(\rho,\Omega)} \to 0,
\quad
k \uparrow \infty.
\end{equation*}
With these convergence properties at hand, 
%\DQ{together with the fact that $J$ is continous on $\mathcal{X} \times \mathbb{U}_{ad}$,} 
we thus conclude the optimality of $\bar{\mathcal{U}}$: $J(\bar{\mathbf{y}},\bar{\mathcal{U}}) = \lim_{k\to\infty}J(\mathbf{y}_{k},\mathcal{U}_{k}) = \mathfrak{i}$. \qed
\end{proof}
%%%%%%%%%%%%%%%%%%%%%%%%%%%%%%%%%%%%%%%%%%%%%%%%%%%%%%%%%%%%
%%%%%%%%%%%%%%%%%%%%%%%%%%%%%%%%%%%%%%%%%%%%%%%%%%%%%%%%%%%%
%%%%%%%%%%%%%%%%%%%%%%%%%%%%%%%%%%%%%%%%%%%%%%%%%%%%%%%%%%%%
%%%%%%%%%%%%%%%%%%%%%%%%%%%%%%%%%%%%%%%%%%%%%%%%%%%%%%%%%%%%
\section{First and Second Order Optimality Conditions} 
\label{sec:op_conditions}

In this section, we analyze first and second order optimality conditions for the optimal control problem \eqref{eq:weak_cost}--\eqref{eq:weak_st_eq}. \EO{We must immediately mention that, since \eqref{eq:weak_cost}--\eqref{eq:weak_st_eq} is not convex, we distinguish between local and global solutions and present optimality conditions in the context of local solutions \cite{MR3936891,MR2338434}.}

\begin{definition}[Local Solutions]
\EO{We say that $(\bar{\mathbf{y}},\bar{\mathcal{U}})$ is a local solution for problem \eqref{eq:weak_cost}--\eqref{eq:weak_st_eq} if there exist neighborhoods $\mathcal{A} \subset \mathbf{H}_{0}^{1}(\rho,\Omega)$ and $\mathcal{B}\subset [\mathbb{R}^{2}]^{\ell} \cap \mathbb{U}_{ad}$ of $\bar{\mathbf{y}}$ and $\bar{\mathcal{U}}$, respectively, such that $J(\bar{\mathbf{y}},\bar{\mathcal{U}}) \leq J(\mathbf{y},\mathcal{U})$ for all $(\mathbf{y},\mathcal{U}) \in \mathcal{A} \times \mathcal{B}$. If the inequality is strict for every $(\mathbf{y},\mathcal{U}) \in \mathcal{A} \times \mathcal{B} \setminus \{  (\bar{\mathbf{y}},\bar{\mathcal{U}})\}$, we say that $(\bar{\mathbf{y}},\bar{\mathcal{U}})$ is a strict local solution.}
\end{definition}

\EO{From now on, we will assume that $(\bar{\mathbf{y}},\bar{\mathcal{U}})$ is a local solution to \eqref{eq:weak_cost}--\eqref{eq:weak_st_eq} such that $\bar{\mathbf{y}}$ is regular. Within this setting, the results of Theorem \ref{thm:diff_prop} guarantee the existence of neighborhoods $\mathcal{O}(\bar{\mathcal{U}}) \subset [\mathbb{R}^{2}]^{\ell} $, $\mathcal{O}(\bar{\mathbf{y}}) \subset \mathbf{V}(\rho,\Omega) $, and $\mathcal{O}(\bar{p}) \subset L^{2}(\rho,\Omega)/\mathbb{R}$ of $ \bar{\mathcal{U}}$, $\bar{\mathbf{y}}$, and $\bar{p}$, respectively, and a map of class $C^2$,
\begin{equation*}
Q: \mathcal{O}(\bar{\mathcal{U}}) \to \mathcal{O}(\bar{\mathbf{y}}) \times \mathcal{O}(\bar{p}),
\end{equation*}
such that $(\bar{\mathbf{y}},\bar{p}) = Q(\bar{\mathcal{U}})$. In addition, for each $\mathcal{U} \in \mathcal{O}(\bar{\mathcal{U}})$, the pair $(\mathbf{y}_{\mathcal{U}},p_{\mathcal{U}}) := Q(\mathcal{U})$ corresponds to the unique solution of \eqref{def:state_eq} in $\mathcal{O}(\bar{\mathbf{y}}) \times \mathcal{O}(\bar{p}).$}
%%%%%%%%%%%%%%%%%%%%%%%%%%%%%%%%%%%%%%%%%%%%%%%%%%%%%%%%%%%%%%%%%
%%%%%%%%%%%%%%%%%%%%%%%%%%%%%%%%%%%%%%%%%%%%%%%%%%%%%%%%%%%%%%%%%
%%%%%%%%%%%%%%%%%%%%%%%%%%%%%%%%%%%%%%%%%%%%%%%%%%%%%%%%%%%%%%%%%
%%%%%%%%%%%%%%%%%%%%%%%%%%%%%%%%%%%%%%%%%%%%%%%%%%%%%%%%%%%%%%%%%
\subsection{Adjoint Equation}
\label{subsec:adjjoint_equation}
\EO{We begin the section by introducing the \emph{adjoint problem} as follows:}
Find $(\mathbf{z},r) \in \mathbf{H}_{0}^{1}(\EO{\rho^{-1}, }\Omega) \times L^{2}(\EO{\rho^{-1}, }\Omega)/\mathbb{R}$ such that
\begin{multline}
\label{eq:weak_adj_eq}
\displaystyle\int_{\Omega} 
\left(\nu \nabla \mathbf{z} : \nabla \mathbf{w} 
-
(\mathbf{y}_\mathcal{U}\cdot \nabla)\mathbf{z}\mathbf{w} 
+ 
\nabla\mathbf{y}_\mathcal{U}^{\intercal}\mathbf{z} \cdot \mathbf{w} 
- 
r\text{div }\mathbf{w}\right) \\
= \displaystyle\int_{\Omega}(\mathbf{y}_\mathcal{U} - \mathbf{y}_{\Omega})\cdot\mathbf{w},
\qquad
\displaystyle \int_{\Omega}  s\text{div }\mathbf{z}= 0,
\end{multline}
for all $(\mathbf{w},s) \in \mathbf{H}_{0}^{1}(\EO{\rho, }\Omega) \times L^{2}(\EO{\rho, }\Omega)/\mathbb{R}$. \EO{Here, $\mathbf{y}_{\mathcal{U}}$ denotes the velocity component of the unique solution $(\mathbf{y}_{\mathcal{U}},p_{\mathcal{U}}) \in \mathbf{H}_{0}^{1}(\rho,\Omega) \times L^{2}(\rho,\Omega)/\mathbb{R}$ to problem \eqref{eq:weak_st_eq}, associated to $\mathcal{U} \in \mathcal{O}(\bar{\mathcal{U}})$, in the neighborhood $\mathcal{O}(\bar{\mathbf{y}}) \times \mathcal{O}(\bar{p})$.}

\EO{The well-posedness of the adjoint problem in weighted spaces is as follows: Since $(\bar{\mathbf{y}},\bar{\mathcal{U}})$ is a local solution to \eqref{eq:weak_cost}--\eqref{eq:weak_st_eq} such that $\bar{\mathbf{y}}$ is regular, a direct application of item (ii) in Theorem \ref{thm:diff_prop} reveals that 
\[
Q'(\mathcal{U}):[\mathbb{R}^{2}]^{\ell} \rightarrow \mathbf{V}(\rho,\Omega) \times L^{2}(\rho,\Omega)/\mathbb{R}
\]
is an isomorphism for every $\mathcal{U} \in \mathcal{O}(\bar{\mathcal{U}})$; the characterization of $Q'(\mathcal{U})$ being available in the item (iii) of Theorem \ref{thm:diff_prop}. On the basis of this fact,  the duality argument elaborated within the proof of Theorem \ref{thm:wellposed_lin_NS_weighted} reveals that problem \eqref{eq:weak_adj_eq} admits a unique solution $(\mathbf{z},r) \in \mathbf{H}_0^1(\rho^{-1},\Omega) \times L^{2}(\rho^{-1},\Omega)/\mathbb{R}$. 
%\DQ{with $\bar{\mathbf{y}}$ replaced by $\mathbf{y}_{\mathcal{U}}$}. 
In addition, in view of \eqref{eq:stability_phi_weighted}, we have the following stability bound in weighted spaces:
\begin{equation}\label{eq:stab_bound_adj}
\begin{aligned}
\| \nabla \mathbf{z} \|_{\mathbf{L}^2(\rho^{-1},\Omega)}
+
\| r\|_{L^2(\rho^{-1},\Omega)} 
& \lesssim 
\| \mathbf{y}_\mathcal{U} - \mathbf{y}_{\Omega} \|_{\mathbf{H}_0^1(\rho,\Omega)'}
\\
& \lesssim 
\| \mathbf{y}_\mathcal{U} - \mathbf{y}_{\Omega} \|_{\mathbf{L}^2(\Omega)}.
\end{aligned}
\end{equation}

The following result guarantees that point evaluations of the velocity component $\mathbf{z}$ of the adjoint pair $(\mathbf{z},r)$ are well-defined.}

\begin{theorem}[Regularity Estimates]\label{thm:reg_estim_adj}
If $(\mathbf{z},r)$ solves \eqref{eq:weak_adj_eq}, then \EO{$\mathbf{z}$ belongs to} $\mathbf{W}^{1,q}(\Omega)$ for some $q > 2$. Consequently, $\mathbf{z} \in \mathbf{C}(\bar \Omega)$.
\end{theorem}
\begin{proof}
\EO{We begin the proof by rewriting the adjoint equation as the following Stokes problem:}
\[
\begin{array}{rcl}
\displaystyle\int_{\Omega} \left(\nu \nabla \mathbf{z} : \nabla \mathbf{w} - r\text{div }\mathbf{w}\right) &=& \displaystyle\int_{\Omega}(\mathbf{y}_\mathcal{U} - \mathbf{y}_{\Omega})\cdot\mathbf{w}
+
\int_{\Omega} 
\left[
(\mathbf{y}_\mathcal{U}\cdot \nabla)\mathbf{z}\mathbf{w} 
- 
\nabla\mathbf{y}_\mathcal{U}^{\intercal}\mathbf{z} \cdot \mathbf{w} 
\right]
,\\
\displaystyle \int_{\Omega} s \text{div }\mathbf{z} &=& 0,
\end{array}
\]
for all \EO{$(\mathbf{w},s) \in \mathbf{H}_{0}^{1}(\rho,\Omega) \times L^{2}(\rho,\Omega)/\mathbb{R}$.}
%$(\mathbf{w},s) \in \mathbf{H}_{0}^{1}(\Omega) \times L^{2}(\Omega)/\mathbb{R}$. 

Denote $\mathbf{W}^{-1,q}(\Omega) = \mathbf{W}_0^{1,q'}(\Omega)'$ and define the linear functional $\mathfrak{F} := \mathfrak{F}_1 - \mathfrak{F}_2$, where $\mathfrak{F}_1, \mathfrak{F}_2: \mathbf{H}_0^1(\EO{\rho},\Omega) \rightarrow \mathbb{R}$ are defined by $\mathfrak{F}_1( \mathbf{w} ):= \int_{\Omega}(\mathbf{y}_\mathcal{U}\cdot \nabla)\mathbf{z}\mathbf{w}$ and $\mathfrak{F}_2( \mathbf{w} ):= \int_{\Omega}\nabla\mathbf{y}_\mathcal{U}^{\intercal}\mathbf{z} \cdot \mathbf{w}$. Let us prove that $\mathfrak{F} \in \mathbf{W}^{-1,q}(\Omega)$ for some $q>2$. To accomplish this task, we first study $\mathfrak{F}_1$ on the basis of \EO{H\"older's inequality:
\begin{equation*}
\begin{aligned}
\| \mathfrak{F}_1 \|_{\mathbf{W}^{-1,q}(\Omega)} 
& \leq \sup_{ \mathbf{w} \in \mathbf{W}_0^{1,q'}(\Omega)} \frac{\|\rho^{\frac{1}{4}}\|_{L^{\infty}(\Omega)} \| \mathbf{y}_{\mathcal{U}} \|_{\mathbf{L}^{4}(\rho,\Omega)} \| \nabla \mathbf{z} \|_{\mathbf{L}^{2}(\rho^{-1},\Omega)}  \| \mathbf{w} \|_{\mathbf{L}^{4}(\Omega)} }{ \| \nabla \mathbf{w} \|_{\mathbf{L}^{q'}(\Omega)}}
\\
& \lesssim
\|\rho^{\frac{1}{4}}\|_{L^{\infty}(\Omega)} \| \nabla \mathbf{y}_{\mathcal{U}} \|_{\mathbf{L}^{2}(\rho,\Omega)} \| \nabla \mathbf{z} \|_{\mathbf{L}^{2}(\rho^{-1},\Omega)},
\end{aligned}
\label{eq:F_1}
\end{equation*}
%\begin{multline}
%\begin{split}
%\| \mathfrak{F}_1 \|_{\mathbf{W}^{-1,q}(\Omega)} 
%\leq \sup_{ \mathbf{w} \in \mathbf{W}_0^{1,q'}(\Omega)} \frac{\| \mathbf{y}_{\mathcal{U}} \|_{\mathbf{L}^{\mu}(\Omega)} \| \nabla \mathbf{z} \|_{\mathbf{L}^{2}(\Omega)}  \| \mathbf{w} \|_{\mathbf{L}^{\kappa}(\Omega)} }{ \| \nabla \mathbf{w} \|_{\mathbf{L}^{q'}(\Omega)}}, \\
%\DQ{\lesssim
%\sup_{ \mathbf{w} \in \mathbf{W}_0^{1,q'}(\Omega)} \frac{\| \mathbf{y}_{\mathcal{U}} \|_{\mathbf{L}^{\mu}(\Omega)} \| \nabla \mathbf{z} \|_{\mathbf{L}^{2}(\rho^{-1},\Omega)}  \| \mathbf{w} \|_{\mathbf{L}^{\kappa}(\Omega)} }{ \| \nabla \mathbf{w} \|_{\mathbf{L}^{q'}(\Omega)}},}
%\end{split}
%\label{eq:F_1}
%\end{multline}
where we have used
%that $\mathbf{z} \in \mathbf{H}_{0}^{1}(\rho^{-1},\Omega)$, and the injections 
$\mathbf{H}_{0}^{1}(\rho,\Omega) \hookrightarrow \mathbf{L}^{4}(\rho,\Omega)$, $\mathbf{W}_{0}^{1,q'}(\Omega)\hookrightarrow \mathbf{L}^{4}(\Omega)$, which holds for $q' \geq 4/3$ $(q\leq4)$, and $\mathbf{z} \in \mathbf{H}_{0}^{1}(\rho^{-1},\Omega)$. We thus deduce the} existence of $q>2$ such that $\mathfrak{F}_1 \in \mathbf{W}^{-1,q}(\Omega)$ and $\| \mathfrak{F}_1 \|_{\mathbf{W}^{-1,q}(\Omega)}  \lesssim \| \nabla \mathbf{y}_{\mathcal{U}}  \|_{\mathbf{L}^{2}(\rho,\Omega)} \| \nabla \mathbf{z} \|_{\mathbf{L}^{2}(\EO{\rho^{-1},}\Omega)}$.

\EO{We now control the term $\mathfrak{F}_{2}$. To accomplish this task, we invoke H\"older's inequality combined with the fact that there exists $\epsilon = \epsilon (\Omega)>0$ such that $\mathbf{y}_{\mathcal{U}} \in \mathbf{W}_{0}^{1,p}(\Omega)$ for every $p \in (4/3 - \epsilon, 2)$: 
%
%and the injections $\mathbf{H}_{0}^{1}(\rho^{-1},\Omega) \hookrightarrow \mathbf{H}_{0}^{1}(\Omega) \hookrightarrow \mathbf{L}^{\beta}(\Omega)$ for all $\beta > 1$ (see Lemma \ref{lemma:density_H-1_dualH01_omega-1}) and $\mathbf{W}_{0}^{1,q'}(\Omega)\hookrightarrow \mathbf{L}^{4}(\Omega)$, with $q' \in (4/3,2)$ to obtain that
\begin{equation*}
\| \mathfrak{F}_{2} \|_{\mathbf{W}^{-1,q}(\Omega)} 
\leq \sup_{ \mathbf{w} \in \mathbf{W}_0^{1,q'}(\Omega)} \frac{\| \nabla \mathbf{y}_{\mathcal{U}} \|_{\mathbf{L}^{p}(\Omega)} \| \mathbf{z} \|_{\mathbf{L}^{\mu}(\Omega)}  \| \mathbf{w} \|_{\mathbf{L}^{\upsilon}(\Omega)} }{ \| \nabla \mathbf{w} \|_{\mathbf{L}^{q'}(\Omega)}},
\label{eq:F_2}
\end{equation*}
with $p^{-1} + \mu^{-1} + \upsilon^{-1} = 1$. Invoke now that $\mathbf{W}_{0}^{1,q'}(\Omega)\hookrightarrow \mathbf{L}^{\sigma}(\Omega)$, which holds for every $\sigma \leq 2q'/(2-q')$, that $\mathbf{z} \in \mathbf{H}_{0}^{1}(\rho^{-1},\Omega)$, and the Sobolev embeddings $\mathbf{H}_{0}^{1}(\rho^{-1},\Omega) \hookrightarrow \mathbf{H}_{0}^{1}(\Omega) \hookrightarrow \mathbf{L}^{\beta}(\Omega)$, which hold for every $\beta < \infty$, to arrive at the existence of $q>2$ such that
$
\| \mathfrak{F}_{2} \|_{\mathbf{W}^{-1,q}(\Omega)} \lesssim \| \nabla \mathbf{y}_{\mathcal{U}} \|_{\mathbf{L}^{p}(\Omega)} \|\nabla \mathbf{z} \|_{\mathbf{L}^{2}(\rho^{-1},\Omega)}.
$

Having} obtained the existence of $q>2$ such that $\mathfrak{F} \in \mathbf{W}^{-1,q}(\Omega)$, it suffices to invoke \cite[(1.52)]{MR2987056} to conclude that $\mathbf{z} \in \mathbf{W}^{1,q}(\Omega)$. \qed
\end{proof}

%%%%%%%%%%%%%%%%%%%%%%%%%%%%%%%%%%%%%%%%%%%%%%%%%%%%%%%%%%%%%%%%%
%%%%%%%%%%%%%%%%%%%%%%%%%%%%%%%%%%%%%%%%%%%%%%%%%%%%%%%%%%%%%%%%%
%%%%%%%%%%%%%%%%%%%%%%%%%%%%%%%%%%%%%%%%%%%%%%%%%%%%%%%%%%%%%%%%%
%%%%%%%%%%%%%%%%%%%%%%%%%%%%%%%%%%%%%%%%%%%%%%%%%%%%%%%%%%%%%%%%%
\subsection{First Order Optimality Conditions}
\label{subsec:first_oder_optimality_conditions}

\EO{In this section, we derive first order optimality conditions for the optimal control problem \eqref{eq:weak_cost}--\eqref{eq:weak_st_eq}. To accomplish this task, we begin this section by introducing some preliminary ingredients. Before presenting them, we recall that we are operating under the assumption that $(\bar{\mathbf{y}},\bar{\mathcal{U}})$ is a local solution to \eqref{eq:weak_cost}--\eqref{eq:weak_st_eq}, which is such that $\bar{\mathbf{y}}$ is regular. The first ingredient is the operator $\mathcal{G}$, which is defined as follows:
\begin{equation}\label{eq:mathcal_G}
\mathcal{G}: \mathcal{O}(\bar{\mathcal{U})} \subset [\mathbb{R}^{2}]^{\ell}
 \rightarrow
  \mathcal{O}(\bar{\mathbf{y}}) \subset \mathbf{H}_{0}^{1}(\rho,\Omega), \qquad \mathcal{U} \mapsto \mathbf{y},
\end{equation}
where $\mathbf{y}$ corresponds to the velocity component of the pair $(\mathbf{y},p) = Q(\mathcal{U})$.}
\EO{The second ingredient is the reduced cost functional:
\begin{equation}\label{def:red_cost_functional}
j: \mathcal{O}(\bar{\mathcal{U}}) \rightarrow \mathbb{R},
\qquad
j(\mathcal{U}):= \dfrac{1}{2}\| \mathcal{G}(\mathcal{U}) - \mathbf{y}_{\Omega} \|^{2}_{\mathbf{L}^{2}(\Omega)} + \dfrac{\eta}{2}\sum_{t \in \mathcal{D}}|\mathbf{u}_{t}|^{2}.
\end{equation}

Having defined the reduced cost functional, we present the following standard result: If $\bar{\mathcal{U}} \in \mathbb{U}_{ad}$ denotes a locally optimal control for problem \eqref{eq:weak_cost}--\eqref{eq:weak_st_eq}, then it satisfies the following variational inequality \cite[Lemma 4.18]{Troltzsch}:
\begin{equation}\label{eq:var_eq}
j'(\bar{\mathcal{U}})(\mathcal{U} - \bar{\mathcal{U}}) \geq 0 \quad \forall \mathcal{U} \in \mathbb{U}_{ad}.
\end{equation}

The following result explores the variational inequality \eqref{eq:var_eq}.}

\begin{theorem}[First Order Optimality Conditions]\label{thm:first_opt_cond}
\EO{If the pair $(\bar{\mathbf{y}},\bar{\mathcal{U}}) \in \mathbf{H}_{0}^{1}(\rho,\Omega) \times \mathbb{U}_{ad}$ denotes a local solution to the optimal control problem \eqref{eq:weak_cost}--\eqref{eq:weak_st_eq} such that $\bar{\mathbf{y}}$ is regular, then $\bar{\mathcal{U}} \in \mathbb{U}_{ad}$ satisfies the variational inequality
\begin{equation}
\sum_{t \in \mathcal{D}}(\bar{\mathbf{z}}(t) + \eta \bar{\mathbf{u}}_{t})\cdot (\mathbf{u}_{t} - \bar{\mathbf{u}}_{t}) \geq 0 \qquad \forall \mathcal{U} = (\mathbf{u}_{1},\ldots,\mathbf{u}_{\ell}) \in \mathbb{U}_{ad},
\label{eq:variational_inequality}
\end{equation}
where $(\bar{\mathbf{z}},\bar{r}) \in \mathbf{H}_{0}^{1}(\rho^{-1},\Omega) \times L^{2}(\rho^{-1},\Omega)/\mathbb{R}$ denotes the optimal adjoint pair, which solves the adjoint problem \eqref{eq:weak_adj_eq} with $\mathbf{y}_{\mathcal{U}}$ replaced by $\bar{\mathbf{y}}$.}
\end{theorem}
\begin{proof}
We begin the proof by computing the expression $j'(\bar{\mathcal{U}})(\mathcal{U} - \bar{ \mathcal{U}})$ and rewriting the basic variational inequality \eqref{eq:var_eq} as follows:
\begin{equation}
\displaystyle\int_{\Omega}(\bar{\mathbf{y}}-\mathbf{y}_{\Omega}) \cdot \mathcal{G}'(\bar{\mathcal{U}}) (\mathcal{U} - \bar{\mathcal{U}})
+ 
\eta\sum_{t\in\mathcal{D}}\bar{\mathbf{u}}_t\cdot(\mathbf{u}_t-\bar{\mathbf{u}}_t)\geq 0
\quad
\forall 
\mathcal{U} \in \mathbb{U}_{ad}.
\label{eq:vi_first}
\end{equation}
%for all $\mathcal{U}=(\mathbf{u}_1,\ldots,\mathbf{u}_l)\in\mathbb{U}_{ad}$. 
Define $\boldsymbol{\theta}:= \mathcal{G}'(\bar{\mathcal{U}}) (\mathcal{U} - \bar{\mathcal{U}})$. Observe that $(\boldsymbol{\theta},\xi) \in \mathbf{H}_{0}^{1}(\rho,\Omega) \times L^{2}(\rho,\Omega)/\mathbb{R}$
%$\mathbf{H}_{0}^{1}(\rho,\Omega) \times L^{2}(\rho,\Omega)/\mathbb{R}$ 
solves
\begin{multline}\label{eq:thm8_eq2}
\displaystyle\int_{\Omega} \left(\nu \nabla \boldsymbol\theta : \nabla \mathbf{v}
-
\bar{\mathbf{y}} \otimes \boldsymbol\theta : \nabla\mathbf{v} 
- 
\boldsymbol{\theta} \otimes \bar{\mathbf{y}} : \nabla\mathbf{v}
- 
\xi\text{div }\mathbf{v}\right) \\
= \displaystyle\sum_{t\in\mathcal{D}}\langle(\mathbf{u}_t - \bar{\mathbf{u}}_t)\delta_{t},\mathbf{v}\rangle,
\qquad
\displaystyle \int_{\Omega} q\text{div }\boldsymbol\theta  = 0,
\end{multline}
for all $(\mathbf{v},q) \in \mathbf{H}_{0}^{1}(\rho^{-1},\Omega) \times L^{2}(\rho^{-1},\Omega)/\mathbb{R}$.
%$\mathbf{H}_{0}^{1}(\rho^{-1},\Omega) \times L^{2}(\rho^{-1},\Omega)/\mathbb{R}$. 
Having introduced the pair $(\boldsymbol{\theta},\xi)$, the variational inequality \eqref{eq:vi_first} becomes
\begin{equation}\label{eq:thm8_eq1}
\displaystyle\int_{\Omega}(\bar{\mathbf{y}}-\mathbf{y}_{\Omega})\cdot\boldsymbol{\theta}+\eta\sum_{t\in\mathcal{D}}\bar{\mathbf{u}}_t\cdot(\mathbf{u}_t-\bar{\mathbf{u}}_t)\geq 0
\qquad
\forall 
\mathcal{U}=(\mathbf{u}_1,\ldots,\mathbf{u}_l)\in\mathbb{U}_{ad}.
\end{equation}
Since the second term on the right-hand side of the previous expression is already present in the  desired inequality \eqref{eq:variational_inequality}, we focus on the first term.

%In view of the \DQ{\emph{Step 2} of Theorem \ref{thm:wellposed_lin_NS_weighted}}, we are allowed to 
\EO{Let us} set $(\bar{\mathbf{z}},\bar{r}) \in \mathbf{H}_{0}^{1}(\rho^{-1},\Omega) \times L^{2}(\rho^{-1},\Omega)/\mathbb{R}$
%$\mathbf{H}_{0}^{1}(\rho^{-1},\Omega) \times L^{2}(\rho^{-1},\Omega)/\mathbb{R}$ 
as a test pair in problem \eqref{eq:thm8_eq2}. This yields
\begin{equation}\label{eq:thm8_eq3}
%\DQ{\displaystyle\int_{\Omega} \nu \nabla \boldsymbol\theta : \nabla  \bar{\mathbf{z}}
%-
%(\bar{\mathbf{y}}\cdot \nabla)\boldsymbol\theta\bar{\mathbf{z}} 
%+ 
%(\nabla \bar{\mathbf{y}})^{\intercal}\boldsymbol{\theta}\cdot\bar{\mathbf{z}}
%= \sum_{t\in\mathcal{D}} (\mathbf{u}_t - \bar{\mathbf{u}}_t) \cdot \bar{\mathbf{z}}(t),}
\displaystyle\int_{\Omega} \left(\nu \nabla \boldsymbol\theta : \nabla \bar{\mathbf{z}}
-
\bar{\mathbf{y}} \otimes \boldsymbol\theta : \nabla\bar{\mathbf{z}} 
- 
\boldsymbol{\theta} \otimes \bar{\mathbf{y}} : \nabla\bar{\mathbf{z}}
\right) 
= \sum_{t\in\mathcal{D}} (\mathbf{u}_t - \bar{\mathbf{u}}_t) \cdot \bar{\mathbf{z}}(t),
\end{equation}
upon utilizing the fact that $\int_{\Omega} \xi \text{div } \bar{\mathbf{z}}$ vanishes and that there exists $q>2$ such that $\bar{\mathbf{z}} \in \mathbf{W}^{1,q}(\Omega) \hookrightarrow \mathbf{C}(\bar \Omega)$ (cf.~Theorem \ref{thm:reg_estim_adj}). \EO{We now set $\mathbf{w} = \boldsymbol\theta$ as a test function in the first equation of problem \eqref{eq:weak_adj_eq} to obtain}
\begin{equation}\label{eq:thm8_eq4}
\displaystyle\int_{\Omega} \left(\nu \nabla  \bar{\mathbf{z}} : \nabla \boldsymbol\theta
-
(\bar{\mathbf{y}}\cdot \nabla)\bar{\mathbf{z}} \boldsymbol\theta
+ 
\nabla \bar{\mathbf{y}}^{\intercal}\bar{\mathbf{z}}\cdot
\boldsymbol{\theta}
- 
\bar{r}\text{div }\boldsymbol{\theta}
\right) = \int_{\Omega}(\bar{\mathbf{y}} - \mathbf{y}_{\Omega}) \cdot \boldsymbol\theta.
\end{equation}
We \EO{thus utilize \eqref{eq:thm8_eq3}, \eqref{eq:thm8_eq4}, the fact that $\int_{\Omega} \bar{r}\text{div }\boldsymbol\theta$ vanishes, and an integration by parts formula for the convective terms in \eqref{eq:thm8_eq4} to arrive at the needed relation $\sum_{t\in\mathcal{D}} (\mathbf{u}_t - \bar{\mathbf{u}}_t) \cdot \bar{\mathbf{z}}(t) = \int_{\Omega}(\bar{\mathbf{y}} - \mathbf{y}_{\Omega}) \cdot \boldsymbol\theta$. In view of \eqref{eq:thm8_eq1}, the previously derived identity allows to arrive at \eqref{eq:variational_inequality}.} \qed
\end{proof}
%
%%%%%%%%%%%%%%%%%%%%%%%%%%%%%%%%%%%%%%%%%%%%%%%%%%%%%%%%%%%%%%%%%
%%%%%%%%%%%%%%%%%%%%%%%%%%%%%%%%%%%%%%%%%%%%%%%%%%%%%%%%%%%%%%%%%
%%%%%%%%%%%%%%%%%%%%%%%%%%%%%%%%%%%%%%%%%%%%%%%%%%%%%%%%%%%%%%%%%
%%%%%%%%%%%%%%%%%%%%%%%%%%%%%%%%%%%%%%%%%%%%%%%%%%%%%%%%%%%%%%%%%
%%
\subsection{Second Order Optimality Conditions}
\label{subsec:second_order_optimality_conditions}

\EO{In this section, we analyze necessary and sufficient second order optimality conditions.}

\subsubsection{\EO{Auxiliary Results}}
\EO{We begin this section by recalling that, as stated at the beginning of Section \ref{sec:op_conditions}, we are operating under the assumption that $(\bar{\mathbf{y}},\bar{\mathcal{U}}) \in \mathbf{H}_{0}^{1}(\rho,\Omega) \times \mathbb{U}_{ad}$ is a local solution of \eqref{eq:weak_cost}--\eqref{eq:weak_st_eq} such that  $\bar{\mathbf{y}}$ is regular.} 
%\DQ{From this section we make the assumption that the neighborhood $\mathcal{O}(\bar{\mathcal{U}})$  is bounded and convex. Otherwise we can restrict the analysis to a new neigborhood contained in $\mathcal{O}(\bar{\mathcal{U}})$ with such properties.}

\EO{We begin our studies with the following estimate.
\begin{lemma}[Auxiliary Estimate]
\label{lemma:auxiliary_estimate}
Let $\mathcal{U}\in \mathcal{O}(\bar{\mathcal{U}})$ and $\mathcal{V}\in[\mathbb{R}^{2}]^{\ell}$. Let $\mathbf{y}=\mathcal{G}(\mathcal{U})$, $(\boldsymbol\theta,\xi) = Q'(\mathcal{U})\mathcal{V}$, and $(\bar{\boldsymbol\theta},\bar{\xi}) = Q'(\bar{\mathcal{U}})\mathcal{V}$. Then, we have the estimate
\begin{equation}\label{eq:estimate_thetas}
\|\nabla (\boldsymbol\theta - \bar{\boldsymbol\theta})\|_{\mathbf{L}^2(\rho,\Omega)} \lesssim \|\mathcal{U} - \bar{\mathcal{U}}\|_{[\mathbb{R}^2]^{\ell}} \|\mathcal{V}\|_{[\mathbb{R}^2]^{\ell}}
\end{equation}
where the hidden constant is independent of $(\boldsymbol\theta,\xi)$, $(\bar{\boldsymbol\theta},\bar{\xi})$, and $\mathcal{V}$.
\end{lemma}}
\begin{proof}
\EO{We begin the proof by noticing that $(\boldsymbol\theta - \bar{\boldsymbol\theta},\xi - \bar{\xi})$
solves the following problem: Find $(\boldsymbol\theta - \bar{\boldsymbol\theta},\xi - \bar{\xi}) \in \mathbf{H}_{0}^{1}(\rho,\Omega) \times L^{2}(\rho,\Omega)/\mathbb{R}$ such that
\begin{multline}\label{eq:bar_theta-theta}
\int_{\Omega} ( \nu \nabla (\boldsymbol\theta - \bar{\boldsymbol\theta}) : \nabla \mathbf{w} - \bar{\mathbf{y}}\otimes (\boldsymbol\theta -\bar{\boldsymbol\theta}) : \nabla \mathbf{w} - (\boldsymbol\theta -\bar{\boldsymbol\theta})  \otimes \bar{\mathbf{y}}: \nabla \mathbf{w} - (\xi - \bar{\xi})\text{div }\mathbf{w} ) \\
= \int_{\Omega} [(\mathbf{y} - \bar{\mathbf{y}})\otimes \boldsymbol\theta + \boldsymbol\theta \otimes (\mathbf{y} - \bar{\mathbf{y}})] : \nabla \mathbf{w}, \quad  \int_{\Omega} s\text{div }(\boldsymbol\theta - \bar{\boldsymbol\theta}) = 0,
\end{multline}
for all $(\mathbf{w},s)\in \mathbf{H}_{0}^{1}(\rho^{-1},\Omega) \times L^{2}(\rho^{-1},\Omega)/\mathbb{R}$. Since $\bar{\mathbf{y}}$ is regular, a direct application of Theorem \ref{thm:wellposed_lin_NS_weighted} reveals that problem \eqref{eq:bar_theta-theta} is well-posed upon realizing that the forcing term of the momentum equation belongs to the dual space of $\mathbf{H}_0^1(\rho^{-1},\Omega)$; see the estimates in \eqref{eq:estimate_hat_theta-theta} below.}
%Since $\bar{\mathbf{y}}$ is regular,
%%and $\mathcal{U}\in \mathcal{O}(\bar{\mathcal{U}})$, 
%it follows, from Theorem \ref{thm:wellposed_lin_NS_weighted}, that problem \eqref{eq:bar_theta-theta} is well posed.}

\EO{To derive \eqref{eq:estimate_thetas} we invoke the stability estimate \eqref{eq:stability_lin_NS_weighted}, H\"older's inequality, the weighted embeddings $\mathbf{H}_0^1(\rho^{\pm},\Omega)\hookrightarrow \mathbf{L}^4(\rho^{\pm},\Omega)$, and the Lipschitz property of Lemma \ref{lemma:lips_prop}. With these arguments and estimates, we have
% for the difference $\mathbf{y} - \bar{\mathbf{y}} = \mathcal{G}(\mathcal{U}) - \mathcal{G}(\bar{\mathcal{U}})$, it follows that
\begin{equation}
\begin{aligned}
\|\nabla(\boldsymbol\theta - \bar{\boldsymbol\theta})\|_{\mathbf{L}^2(\rho,\Omega)}
&  \lesssim 
%\DQ{(1 + \|\nabla \mathbf{y}\|_{\mathbf{L}^2(\rho,\Omega)} + \|\nabla \bar{\mathbf{y}}\|_{\mathbf{L}^2(\rho,\Omega)})}
\|(\mathbf{y} - \bar{\mathbf{y}}) \otimes \boldsymbol\theta  + \bar{\boldsymbol\theta} \otimes (\mathbf{y} - \bar{\mathbf{y}})\|_{\mathbf{H}_0^1(\rho^{-1},\Omega)'}
\\
& \lesssim
%\DQ{(1 + \|\nabla \mathbf{y}\|_{\mathbf{L}^2(\rho,\Omega)} + \|\nabla \bar{\mathbf{y}}\|_{\mathbf{L}^2(\rho,\Omega)})}
\left[ \|\nabla\boldsymbol\theta\|_{\mathbf{L}^2(\rho,\Omega)} + \|\nabla\bar{\boldsymbol\theta}\|_{\mathbf{L}^2(\rho,\Omega)} \right]
\|\nabla(\mathbf{y} - \bar{\mathbf{y}})\|_{\mathbf{L}^2(\rho,\Omega)}
\\
& \lesssim
\|\mathcal{V}\|_{[\mathbb{R}^2]^{\ell}}
 \|\mathcal{U}-\bar{\mathcal{U}}\|_{[\mathbb{R}^2]^{\ell}}.
\end{aligned}
\label{eq:estimate_hat_theta-theta}
\end{equation}
This concludes the proof. \qed}
\end{proof}

\EO{We conclude this section with the following result.}

\begin{theorem}[Properties of $j''$]\label{thm:property_j''} 
\EO{The reduced cost functional $j$, defined in \eqref{def:red_cost_functional}, is of class $C^2$. In addition, for $\mathcal{U}\in \mathcal{O}(\bar{\mathcal{U}})$ and $\mathcal{V}\in[\mathbb{R}^{2}]^{\ell}$, we have}
\begin{equation}\label{eq:characterization_j''}
\EO{j''(\mathcal{U})\mathcal{V}^{2} = \int_{\Omega}|\boldsymbol\theta|^{2} + 2\int_{\Omega}\boldsymbol\theta\otimes\boldsymbol\theta:\nabla \mathbf{z} + \eta\sum_{t\in\mathcal{D}}|\mathbf{v}_{t}|^{2}.}
\end{equation}
\EO{Here, $(\boldsymbol\theta,\xi) = Q'(\mathcal{U})\mathcal{V}$ and $(\mathbf{z},r)$ denotes the unique solution of \eqref{eq:weak_adj_eq} with $\mathbf{y}_{\mathcal{U}} = \mathcal{G}(\mathcal{U})$. Finally, we have the bound}
\begin{equation}\label{eq:j''_lipschitz}
\EO{|j''(\mathcal{U})\mathcal{V}^2-j''(\bar{\mathcal{U}})\mathcal{V}^2|\lesssim \|\mathcal{U}-\bar{\mathcal{U}}\|_{[\mathbb{R}^2]^{\ell}}\|\mathcal{V}\|_{[\mathbb{R}^2]^{\ell}}^2.}
\end{equation}
\end{theorem}
\begin{proof}
\EO{Since Theorem \ref{thm:diff_prop} guarantees that $Q$ is second order Fr\'echet differentiable, it is immediate that $\mathcal{G}$, defined in \eqref{eq:mathcal_G}, is second order Fr\'echet differentiable as a map from $[\mathbb{R}^{2}]^{\ell}$ into $\mathbf{H}_0^1(\rho,\Omega)$. Consequently, $j$ is of class $C^2$.} 

\EO{We now derive the identity \eqref{eq:characterization_j''}. To accomplish this task, we begin with a basic computation, which reveals that, for $\mathcal{U}\in \mathcal{O}(\bar{\mathcal{U}})$ and $\mathcal{V} \in[\mathbb{R}^2]^{\ell}$, we have
\begin{equation}\label{eq:charac_j2_prev}
j''(\mathcal{U})\mathcal{V}^{2}
=
\int_{\Omega} |\mathcal{G}'(\mathcal{U})\mathcal{V}|^{2} + \int_{\Omega}(\mathcal{G}(\mathcal{U})-\mathbf{y}_{\Omega})\cdot \mathcal{G}''(\mathcal{U})\mathcal{V}^{2} + \eta\sum_{t\in\mathcal{D}}|\mathbf{v}_{t}|^{2}.
\end{equation}
Define $(\boldsymbol \psi, \gamma) := Q''(\mathcal{U})\mathcal{V}^{2}$, i.e., $(\boldsymbol \psi, \gamma)  \in \mathbf{H}_{0}^{1}(\rho,\Omega) \times L^{2}(\rho,\Omega)/\mathbb{R}$ corresponds to the unique solution to \eqref{eq:second_der_sol_map_S} with both $\boldsymbol\theta_{\mathcal{V}_{1}}$ and $\boldsymbol\theta_{\mathcal{V}_{2}}$ being replaced by $\boldsymbol\theta$ and $\mathbf{y}_{\mathcal{U}} =  \mathcal{G}(\mathcal{U})$. Notice that $\boldsymbol \psi = \mathcal{G}''(\mathcal{U})\mathcal{V}^{2}$. Setting $\mathbf{v} = \mathbf{z}$ in the first equation of the problem that $(\boldsymbol \psi, \gamma)$ solves (c.f. \eqref{eq:second_der_sol_map_S} with  $\boldsymbol\theta_{\mathcal{V}_{1}} = \boldsymbol\theta_{\mathcal{V}_{2}} =\boldsymbol\theta$), we arrive at
\[
\int_{\Omega} ( [\nu \nabla \boldsymbol{\psi} - 
\mathbf{y}_{\mathcal{U}} \otimes \boldsymbol{\psi}
-
\boldsymbol{\psi} \otimes \mathbf{y}_{\mathcal{U}}]: \nabla\mathbf{z}
- 
\gamma\textnormal{div }\mathbf{z}) = 2\int_{\Omega}\boldsymbol{\theta} \otimes \boldsymbol{\theta} : \nabla \mathbf{z}.
\]
Here, $(\mathbf{z},r) \in \mathbf{H}_{0}^{1}(\rho^{-1},\Omega) \times L^{2}(\rho^{-1},\Omega)/\mathbb{R}$ corresponds to the unique solution to the adjoint problem \eqref{eq:weak_adj_eq}. Similarly, we set $\mathbf{w} = \boldsymbol\psi$ in the first equation of the adjoint problem \eqref{eq:weak_adj_eq}. This yields
\[
\int_{\Omega} ( \nu \nabla \mathbf{z}:\nabla \mathbf{\boldsymbol\psi}
 - 
(\mathbf{y}_{\mathcal{U}} \cdot \nabla) \mathbf{z} \boldsymbol{\psi}
+
\nabla \mathbf{y}_{\mathcal{U}}^{\intercal}\mathbf{z}\cdot \boldsymbol{\psi} - r\text{div }\boldsymbol{\psi})
= 
\int_{\Omega}(\mathbf{y}_{\mathcal{U}} - \mathbf{y}_{\Omega}) \cdot \boldsymbol{\psi}.
 \]
We now resort to an integration by parts argument based on the  fact that $\text{div} \ \boldsymbol\psi = \text{div} \ \mathbf{y}_{\mathcal{U}} = \text{div} \ \mathbf{z} = 0$ to obtain \eqref{eq:characterization_j''}.}

\EO{Let us proceed with the task of deriving \eqref{eq:j''_lipschitz}. To accomplish this task, we define $(\bar{\boldsymbol\theta},\bar{\xi}) := Q'(\bar{\mathcal{U}})\mathcal{V}$ and notice that $(\boldsymbol\theta,\xi)$ and $(\bar{\boldsymbol\theta},\bar{\xi})$ solve problem \eqref{eq:lin_NS_weighted} with $\boldsymbol\Phi = \mathcal{G}(\mathcal{U})$ and $\boldsymbol\Phi = \mathcal{G}(\bar{\mathcal{U}})$, respectively, and $\mathbf{g} = \sum_{t \in \mathcal{D}} \mathbf{v}_{t}\delta_{t} $. In view of the derived identity \eqref{eq:characterization_j''}, we write the following equality:
\begin{multline}
\label{eq:identity_ju_jhatu}
j''(\mathcal{U})\mathcal{V}^2-j''(\bar{\mathcal{U}})\mathcal{V}^2
=
\int_{\Omega}(\boldsymbol\theta - \bar{\boldsymbol\theta}) \cdot (\boldsymbol\theta + \bar{\boldsymbol\theta}) + 2\left[\int_{\Omega} \boldsymbol\theta\otimes (\boldsymbol\theta - \bar{\boldsymbol\theta}) : \nabla \mathbf{z}\right.
\\
\left. + \int_{\Omega} \boldsymbol\theta\otimes \bar{\boldsymbol\theta} : \nabla (\mathbf{z} - \bar{\mathbf{z}}) + \int_{\Omega} (\boldsymbol\theta - \bar{\boldsymbol\theta})\otimes \bar{\boldsymbol\theta} : \nabla \bar{\mathbf{z}}\right].
\end{multline}
Here, $(\bar{\mathbf{z}},\bar{r})$ denotes the unique solution to problem \eqref{eq:weak_adj_eq} with $\mathbf{y}_{\mathcal{U}} $ being replaced by $\bar{\mathbf{y}} = \mathcal{G}(\bar{\mathcal{U}})$. Invoke H\"older's inequality and the Sobolev embeddings $\mathbf{H}_0^1(\rho,\Omega)\hookrightarrow \mathbf{L}^4(\rho,\Omega)$ (cf.~\cite[Theorem 1.3]{Fabes_et_al1982}) and $\mathbf{H}_0^1(\rho,\Omega)\hookrightarrow \mathbf{L}^2(\Omega)$ (cf.~Theorem \ref{thm:embedding_result}) to arrive at
\begin{multline*}
|j''(\mathcal{U})\mathcal{V}^2-j''(\bar{\mathcal{U}})\mathcal{V}^2|
\lesssim
\| \nabla \boldsymbol\theta\|_{\mathbf{L}^{2}(\rho,\Omega)}\| \nabla \bar{\boldsymbol\theta}\|_{\mathbf{L}^{2}(\rho,\Omega)}\|\nabla(\mathbf{z} - \bar{\mathbf{z}})\|_{\mathbf{L}^2(\rho^{-1},\Omega)}
\\
+ \Lambda \|\nabla(\boldsymbol\theta - \bar{\boldsymbol\theta})\|_{\mathbf{L}^2(\rho,\Omega)}
\big(\| \nabla \boldsymbol\theta\|_{\mathbf{L}^2(\rho,\Omega)} + \| \nabla \bar{\boldsymbol\theta}\|_{\mathbf{L}^2(\rho,\Omega)} \big).
\end{multline*}
where $\Lambda:= 1 + \|\nabla \mathbf{z}\|_{\mathbf{L}^2(\rho^{-1},\Omega)} +  \|\nabla \bar{\mathbf{z}}\|_{\mathbf{L}^2(\rho^{-1},\Omega)}$. This bound combined with the stability estimate \eqref{eq:stability_lin_NS_weighted}, the auxiliary estimate \eqref{eq:estimate_thetas}, and the boundedness of $\bar{\mathbf{z}}, \mathbf{z}$ in $\mathbf{H}_0^1(\rho^{-1},\Omega)$, which follows from \eqref{eq:stab_bound_adj}, yield}
\begin{equation*}
|j''(\mathcal{U})\mathcal{V}^2-j''(\bar{\mathcal{U}})\mathcal{V}^2|
\lesssim \|\mathcal{U} - \bar{\mathcal{U}}\|_{[\mathbb{R}^2]^{\ell}} \|\mathcal{V}\|_{[\mathbb{R}^2]^{\ell}}^2 + \|\nabla(\mathbf{z} - \bar{\mathbf{z}})\|_{\mathbf{L}^2(\rho^{-1},\Omega)}\|\mathcal{V}\|_{[\mathbb{R}^2]^{\ell}}^2.
\end{equation*}
Therefore, it suffices to bound $\|\nabla(\mathbf{z} -\bar{\mathbf{z}})\|_{\mathbf{L}^2(\EO{\rho^{-1}},\Omega)}$. To accomplish this goal, we notice that $(\mathbf{z} - \bar{\mathbf{z}}, r - \bar{r})\in \mathbf{H}_0^1(\EO{\rho^{-1}},\Omega)\times L^2(\EO{\rho^{-1}},\Omega)/\mathbb{R}$ solves
\begin{multline}\label{eq:z_hatz}
\displaystyle\int_{\Omega} 
\left(\nu \nabla (\mathbf{z} - \bar{\mathbf{z}}) : \nabla \mathbf{w} 
-
(\mathbf{y}_\mathcal{U}\cdot \nabla)(\mathbf{z} - \bar{\mathbf{z}})\mathbf{w} 
+ 
\nabla\mathbf{y}_\mathcal{U}^{\intercal}(\mathbf{z} - \bar{\mathbf{z}})\cdot \mathbf{w} 
- 
(r - \bar{r})\text{div }\mathbf{w}\right) \\
= \displaystyle\int_{\Omega}(\mathbf{y}_\mathcal{U} - \EO{\bar{\mathbf{y}}})\cdot\mathbf{w} + \int_{\Omega}([\mathbf{y}_\mathcal{U} - \EO{\bar{\mathbf{y}}}]\cdot \nabla) \bar{\mathbf{z}}\mathbf{w} + \int_{\Omega}[\nabla(\EO{\bar{\mathbf{y}}} - \mathbf{y}_\mathcal{U})]^{\intercal}\bar{\mathbf{z}}\cdot \mathbf{w},
\\
\displaystyle \int_{\Omega}  s\text{div}(\mathbf{z} - \bar{\mathbf{z}})= 0, 
\end{multline}
for all $(\mathbf{w},s) \in \mathbf{H}_{0}^{1}(\EO{\rho},\Omega) \times L^{2}(\EO{\rho},\Omega)/\mathbb{R}$. \EO{We now bound the $\mathbf{H}_{0}^{1}(\rho,\Omega)'$-norm of the right hand-side of the first equation of \eqref{eq:z_hatz}. We begin by noticing that H\"older's inequality combined with the embedding $\mathbf{H}_{0}^{1}(\rho,\Omega) \hookrightarrow \mathbf{L}^{4}(\rho,\Omega)$ yield
\[
\|([\mathbf{y}_\mathcal{U} - \bar{\mathbf{y}}]\cdot \nabla) \bar{\mathbf{z}}\|_{\mathbf{H}_{0}^{1}(\rho,\Omega)'} \lesssim \|\nabla(\mathbf{y}_\mathcal{U} - \EO{\bar{\mathbf{y}}})\|_{\mathbf{L}^{2}(\rho,\Omega)}\|\nabla\bar{\mathbf{z}}\|_{\mathbf{L}^{2}(\rho^{-1},\Omega)}.
\]
Similarly, by exploiting the fact that $\EO{\bar{\mathbf{y}}, \mathbf{y}_{\mathcal{U}}} \in \mathbf{W}_0^{1,p}(\Omega)$ for every $p \in (4/3 - \epsilon, 2)$, where $\epsilon = \epsilon(\Omega) > 0$ (cf. estimate \eqref{eq:NS_p_stab}),  we obtain
\[
\|[\nabla(\mathbf{y}_\mathcal{U} - \bar{\mathbf{y}})]^{\intercal} \bar{\mathbf{z}}\|_{\mathbf{H}_{0}^{1}(\rho,\Omega)'} 
\lesssim
\sup_{\mathbf{v} \in \mathbf{H}_{0}^{1}(\rho,\Omega) }
% \|\rho^{-1/4}\|_{\mathbf{L}^{\zeta}(\Omega)}
\frac{
 \|\nabla(\mathbf{y}_\mathcal{U} - \bar{\mathbf{y}})\|_{\mathbf{L}^{p}(\Omega)}
 \|\bar{\mathbf{z}}\|_{\mathbf{L}^{\mu}(\Omega)}
  \|\mathbf{v}\|_{\mathbf{L}^{\varrho}(\Omega)}}
  {\| \nabla\mathbf{v} \|_{\mathbf{L}^{2}(\rho,\Omega)}}, 
\]
where $p^{-1} + \mu^{-1} + \varrho^{-1} = 1$. In view of the embeddings $\mathbf{H}_{0}^{1}(\rho^{-1},\Omega) \hookrightarrow \mathbf{H}_{0}^{1}(\Omega)  \hookrightarrow  \mathbf{L}^{\beta}(\Omega)$ for every $\beta < \infty$, we can thus set $\varrho = 2 + \varepsilon$, with $\varepsilon$ being dictated by Theorem \ref{thm:embedding_result} and utilize the embedding $ \mathbf{H}_0^1(\EO{\rho},\Omega) \hookrightarrow \mathbf{L}^{2 + \varepsilon}(\Omega)$ to conclude that 
\[
\|[\nabla(\mathbf{y}_\mathcal{U} - \bar{\mathbf{y}})]^{\intercal} \bar{\mathbf{z}}\|_{\mathbf{H}_{0}^{1}(\rho,\Omega)'} 
\lesssim
 \|\nabla(\mathbf{y}_\mathcal{U} - \bar{\mathbf{y}})\|_{\mathbf{L}^{p}(\Omega)}
 \|\nabla \bar{\mathbf{z}}\|_{\mathbf{L}^{2}(\rho^{-1},\Omega)}. 
\]
Having controlled the $\mathbf{H}_{0}^{1}(\rho,\Omega)'$-norm of the right hand-side of the first equation of \eqref{eq:z_hatz}, we invoke the weighted stability estimate \eqref{eq:stab_bound_adj} twice to obtain
\begin{multline*}
\| \nabla (\mathbf{z} - \bar{\mathbf{z}}) \|_{\mathbf{L}^2(\rho^{-1},\Omega)}
\lesssim  
\left( \| \nabla (\mathbf{y}_{\mathcal{U}} - \EO{\bar{\mathbf{y}}}) \|_{\mathbf{L}^2(\rho,\Omega)}
+
 \|\nabla(\mathbf{y}_\mathcal{U} - \bar{\mathbf{y}})\|_{\mathbf{L}^{p}(\Omega)}
 \right)
 \\
 \cdot
 \|\nabla \bar{\mathbf{z}}\|_{\mathbf{L}^{2}(\rho^{-1},\Omega)}
 \lesssim 
 \left( \| \nabla (\mathbf{y}_{\mathcal{U}} - \EO{\bar{\mathbf{y}}}) \|_{\mathbf{L}^2(\rho,\Omega)}
+
 \|\nabla(\mathbf{y}_\mathcal{U} - \bar{\mathbf{y}})\|_{\mathbf{L}^{p}(\Omega)}
 \right) \| \bar{\mathbf{y}} - \mathbf{y}_{\Omega}  \|_{\mathbf{L}^2(\Omega)}.
\end{multline*}
%The desired estimate \eqref{eq:j''_lipschitz} thus follows from 
We now utilize the Lipschitz property of Lemma \ref{lemma:lips_prop} and the one in \cite[Lemma 4.4]{MR3936891}, upon further restricting the neighborhood $\mathcal{O}(\bar{\mathcal{U}})$ if necessary, to obtain the bound $\| \nabla (\mathbf{z} - \bar{\mathbf{z}}) \|_{\mathbf{L}^2(\rho^{-1},\Omega)} \lesssim \|\mathcal{U} - \bar{\mathcal{U}}\|_{[\mathbb{R}^2]^{\ell}}$. This concludes the proof.} \qed
\end{proof}

\subsubsection{Second Order Necessary \EO{and Sufficient} Optimality Conditions}

Before \EO{presenting necessary and sufficient second order optimality conditions, we introduce a few ingredients. Let us define}
\begin{equation}\label{eq:derivative_j}
\boldsymbol\Psi:=(\boldsymbol\Psi_{1},\ldots,\boldsymbol\Psi_{\ell})\in [\mathbb{R}^2]^{\ell}, 
\qquad
\boldsymbol\Psi_{t}:=\bar{\mathbf{z}}(t) + \eta\bar{\mathbf{u}}_{t}, 
\qquad 
 t\in \mathcal{D}.
\end{equation}
Let $s \in \mathcal{D}$ and $\mathcal{U} \in \mathbb{U}_{ad}$ be such that $\mathbf{u}_t = \bar{\mathbf{u}}_{t}$ for $t \in \mathcal{D} \setminus \{s\}$. Set $\mathcal{U}$ into the variational inequality \eqref{eq:variational_inequality}. This yields
\begin{equation}
0 \leq (\bar{\mathbf{z}}(s) + \eta \bar{\mathbf{u}}_s) \cdot (\mathbf{u}_s - \bar{\mathbf{u}}_s)  
= 
\boldsymbol\Psi_{s} \cdot (\mathbf{u}_s - \bar{\mathbf{u}}_s).
\label{eq:var_ineq_aux}
\end{equation}
Let $i, j \in \{1,2\}$ be such that $i \neq j$. Set $(\mathbf{u}_s)_{i} = (\bar{\mathbf{u}}_s)_i$. If $ (\bar{\mathbf{u}}_s)_j = (\mathbf{a}_s)_j$, then inequality \eqref{eq:var_ineq_aux} reveals that
\[
0 \leq \boldsymbol\Psi_{s} \cdot (\mathbf{u}_s - \bar{\mathbf{u}}_s) = (\boldsymbol\Psi_{s})_j [ (\mathbf{u}_s)_j -  (\mathbf{a}_s)_j ] \implies (\boldsymbol\Psi_{s})_j  \geq 0.
\]
Similarly, 
\begin{itemize}
\item if $(\mathbf{a}_{s})_{j} < (\bar{\mathbf{u}}_{s})_{j} < (\mathbf{b}_{s})_{j}$, then
$(\boldsymbol\Psi_{s})_{j} = 0$, and
\item if $(\bar{\mathbf{u}}_{s})_{j} = (\mathbf{b}_{s})_{j}$, then
$(\boldsymbol\Psi_{s})_{j} \leq 0$.
\end{itemize}
\EO{Let us also} introduce the cone of critical directions at $\bar{\mathcal{U}} \in \mathbb{U}_{ad}$:
\begin{equation*}
\mathbf{C}_{\bar{\mathcal{U}}}:=\{\mathcal{V}=(\mathbf{v}_{1},\ldots,\mathbf{v}_{\ell})\in [\mathbb{R}^2]^{\ell}\, \text{ that satisfies } \eqref{eq:sign_cond} \text{ and } \eqref{eq:critical_cone_charac_2} \},
\end{equation*}
where, for $t\in\mathcal{D}$ and $i\in\{1,2\}$, conditions \eqref{eq:sign_cond} and \eqref{eq:critical_cone_charac_2} read as \EO{follows:
\begin{equation}\label{eq:sign_cond}
(\mathbf{v}_{t})_{i}
\begin{cases}
\geq 0  \quad \text{ if } (\bar{\mathbf{u}}_{t})_{i} = (\mathbf{a}_{t})_{i},
%\text{ and } (\boldsymbol\Psi_{t})_{i} = 0, 
\\
\leq 0  \quad \text{ if } (\bar{\mathbf{u}}_{t})_{i} = (\mathbf{b}_{t})_{i}, 
%\text{ and } (\boldsymbol\Psi_{t})_{i} = 0,
\end{cases}
\end{equation}
and
\begin{equation}\label{eq:critical_cone_charac_2}
(\boldsymbol\Psi_{t})_i\neq 0 \implies (\mathbf{v}_{t})_{i}=0;
\end{equation}
compare with \cite[(3.16)]{MR3311948}.}
%\subsubsection{Second Order Necessary Optimality Conditions}
%\DQ{At this point it is important to notice that the }

%We now present a second order necessary optimality condition. 

\EO{As stated in \cite[Section 3.3]{MR3311948}, the following result follows from the standard Karush--Kuhn--Tucker theory of mathematical optimization in finite-dimensional spaces (see, for instance, \cite[Theorem 3.8]{MR3311948} and \cite[Section 6.3]{MR2012832}) on the basis of the results derived in Theorem \ref{thm:property_j''}.} 

\begin{theorem}[Second Order Necessary \EO{and Sufficient} Optimality Conditions]\label{thm:second_necess_suff_opt_cond} 
\EO{If $\bar{\mathcal{U}} \in \mathbb{U}_{ad}$ is a local minimum for problem \eqref{eq:weak_cost}--\eqref{eq:weak_st_eq}, then $j''(\bar{\mathcal{U}})\mathcal{V}^2 \geq 0$ for all $\mathcal{V}\in \mathbf{C}_{\bar{\mathcal{U}}}$.} \EO{Conversely, if $\bar{\mathcal{U}} \in \mathbb{U}_{ad}$ satisfies the variational inequality \eqref{eq:variational_inequality} and the second order sufficient condition
\begin{equation}
\label{eq:second_order_sufficient}
j''(\bar{\mathcal{U}})\mathcal{V}^2 > 0 \quad \forall \mathcal{V}\in \mathbf{C}_{\bar{\mathcal{U}}} \setminus \{ \mathbf{0} \},
\end{equation}
then there exists $\mu>0$ and $\sigma >0$ such that
\begin{equation}\label{eq:quad_growth_new}
\displaystyle j(\mathcal{U})\geq j(\bar{\mathcal{U}})+\frac{\mu}{2}\|\mathcal{U}-\bar{\mathcal{U}}\|_{[\mathbb{R}^2]^\ell}^2
\qquad \forall \mathcal{U}\in \mathbb{U}_{ad}: \|\mathcal{U}-\bar{\mathcal{U}}\|_{[\mathbb{R}^2]^\ell}\leq \sigma.
\end{equation}
In particular, $\bar{\mathcal{U}} $ is a strict local solution of \eqref{eq:weak_cost}--\eqref{eq:weak_st_eq}.
}
\end{theorem}

\EO{To present the following result, we introduce, for $\tau > 0$, the cone 
\begin{equation*}
\mathbf{C}_{\bar{\mathcal{U}}}^{\tau}:=\{\mathcal{V}=(\mathbf{v}_{1},\ldots,\mathbf{v}_{\ell})\in [\mathbb{R}^2]^{\ell} \text{ that satisfies } \eqref{eq:sign_cond} \text{ and } \eqref{eq:v_i_tau} \},
\end{equation*}
where, for $t \in \mathcal{D}$ and $i \in \{ 1, 2 \}$, condition \eqref{eq:v_i_tau} reads as follows:}
%\begin{equation}\label{eq:sign_cond_tau}
%(\mathbf{v}_{t})_{i}
%\begin{cases}
%\geq 0  \quad \text{ if } (\bar{\mathbf{u}}_{t})_{i} = (\mathbf{a}_{t})_{i} \text{ and } |(\boldsymbol\Psi_{t})_{i}| \leq \tau, \\
%\leq 0  \quad \text{ if } (\bar{\mathbf{u}}_{t})_{i} = (\mathbf{b}_{t})_{i} \text{ and } |(\boldsymbol\Psi_{t})_{i}| \leq \tau,
%\end{cases}
%\end{equation}
%and
\begin{equation}\label{eq:v_i_tau}
 |(\boldsymbol\Psi_{t})_i| > \tau \implies (\mathbf{v}_{t})_{i} = 0.
\end{equation}

\EO{The following result is immediate in view of our finite dimensional setting.}

\begin{theorem}[\EO{Equivalence}]\label{thm:equivalent_opt_cond}
\EO{Let $(\bar{\mathbf{y}},\bar{p})$, $(\bar{\mathbf{z}},\bar{r})$ and $\bar{\mathcal{U}}$ satisfy 
the first order optimality conditions \eqref{eq:weak_st_eq}, \eqref{eq:weak_adj_eq}, and \eqref{eq:variational_inequality}. Then, \eqref{eq:second_order_sufficient} is equivalent to
%the following statements are equivalent:
%\begin{equation}\label{eq:second_order_2_2}
%j''(\bar{\mathcal{U}})\mathcal{V}^2 > 0 \quad \forall \mathcal{V} \in \mathbf{C}_{\bar{\mathcal{U}}}\setminus \{\mathbf{0}\},
%\end{equation}
%and
}
\begin{equation}\label{eq:second_order_equivalent}
\exists \kappa, \tau >0: \quad j''(\bar{\mathcal{U}})\mathcal{V}^2  \geq \kappa \|\mathcal{V}\|_{[\mathbb{R}^2]^{\ell}}^2 \quad \forall \mathcal{V} \in \mathbf{C}_{\bar{\mathcal{U}}}^\tau.
\end{equation}
\end{theorem}

\section{Conclusions}\label{sec:conclusions}

\EO{We have analyzed, in two dimensions, an optimal control problem for the Navier--Stokes equations within a functional framework inherited by Muckenhoupt weights and Muckenhoupt weighted Sobolev spaces. The control variable corresponds to the amplitude of forces modeled as point sources. In view of the nonuniqueness of solutions for the Navier--Stokes equations, we have operated under an assumption guaranteeing local uniqueness of the state equation around optimal controls. With this assumption at hand, we were able to prove that a suitable linearization of the Navier--Stokes equations is well-posed in weighted spaces. We have also proved that the adjoint state equation is well-posed in weighted spaces and derived further regularity properties for its solution. A combination of these results allowed us to obtain first order necessary optimality conditions. Finally, we provided second order necessary and sufficient optimality conditions on suitable cones of critical directions.

This contribution can be seen as a first step for the development of numerical methods for approximating solutions to the aforementioned optimal control problem. In particular, the obtained results could be useful to obtain a priori and a posteriori error estimates for suitable finite element discretizations.}

\begin{acknowledgements}
FF is supported by UTFSM through Beca de Mantenci\'on. FL is partially supported by DIUBB through project 2120173 GI/C Universidad del B\'io--B\'io and ANID through FONDECYT grant 11200529. EO is partially supported by ANID through FONDECYT grant 1220156. DQ is partially supported by ANID/Subdirecci\'on del Capital Humano/Doctorado Nacional/2021--21210988.
\end{acknowledgements}

\footnotesize
\bibliographystyle{acm}
\bibliography{references_sin_url}

\end{document}